\documentclass[preprint,12pt]{elsarticle}

\usepackage{amsmath}
\usepackage{amsfonts}
\usepackage{amssymb}
\usepackage{amsthm}
\usepackage{booktabs}
\usepackage{ragged2e}
\usepackage[utf8]{inputenc}
\usepackage[english]{babel}
\usepackage{graphicx}
 \usepackage{parskip}
\usepackage{textpos}
\usepackage[utf8]{inputenc}
\usepackage[T1]{fontenc}
\usepackage{helvet}
\usepackage{amsfonts}
\usepackage{amsmath}
\usepackage{amssymb}
\usepackage{tikz, graphicx}
\usepackage{hyperref}
\biboptions{numbers,sort&compress}
\usepackage{bigints}
\usepackage{xcolor}
\usepackage{subfig}     
\usepackage{float}     
\usepackage{geometry}

\newtheorem{theorem}{Theorem}[section]

\newtheorem{definition}[theorem]{Definition}
\newtheorem{corollary}[theorem]{Corollary}

\newtheorem{lemma}[theorem]{Lemma}

\newtheorem{remark}[theorem]{Remark}

\newtheorem{proposition}{Proposition}

\usepackage{float}
\usepackage{graphicx} 
\numberwithin{equation}{section}

\newcommand{\be}{\begin{equation}}
\newcommand{\ee}{\end{equation}}

\begin{document}
\begin{frontmatter}



\title{Taylor-Accelerated Neural Network Interpolation Operators on Irregular Grids: Higher-Order Approximation and Derivative-Informed Field Reconstruction} 


\author{Sachin Saini}
\ead{sachin\_s@ma.iitr.ac.in; sbisaini003@gmail.com}

\affiliation{organization={Department of Mathematics, Indian Institute of Technology Roorkee},
            city={Roorkee},
            postcode={247667},
            state={Uttarakhand},
            country={India}}

\begin{abstract}
In this paper, a new class of Taylor-accelerated neural-network interpolation operators is introduced for quasi-uniform, irregularly distributed grids to reconstruct derivative-informed physical fields in both univariate and multivariate settings. The framework incorporates Taylor expansions within the neural network operators framework, allowing operators to efficiently utilize the smoothness of the target function. We determine the analytical properties of the proposed operators, including well-posedness, uniform boundedness, and exact interpolation at the sampling nodes. Additionally, we establish polynomial reproduction to a specified degree and provide Jackson-type error estimates based on higher-order moduli of smoothness, thereby improving convergence rates on quasi-uniform, irregular grids. The applicability of the method is demonstrated by reconstructing physical fields such as temperature, electromagnetic fields, and fluid properties, including pressure and velocity. The numerical results show that the proposed operators provide reliable, accurate estimates in the presence of noise and irregular sampling. The enhanced accuracy of the proposed method is validated by comparing it with current interpolation systems on irregular grids, including Lagrange-type methods. 
\end{abstract}

\end{frontmatter}

\section{Introduction and Preliminaries}
In recent years, artificial intelligence (AI), particularly neural network-based methods, has become increasingly important for addressing complex engineering problems involving function approximation, data reconstruction, and scientific computing. In many practical situations, such as fluid dynamics, thermal analysis, and electromagnetic signal processing, it is necessary to recover continuous physical fields from data that are often sparse and irregularly distributed. This reconstruction task becomes even more challenging when the available observations are noisy and nonuniform.

In contrast to purely data-driven approaches, the framework of neural network operators (NNOs) provides a more constructive and mathematically transparent way to perform approximation. By interpreting neural networks as operators acting on function spaces, this approach allows rigorous study of convergence properties and the derivation of explicit error estimates using tools from classical approximation theory. From this viewpoint, NNOs can be seen as an alternative to optimization-based models, offering the additional benefit of explicit approximation formulas along with clear theoretical guarantees.

The systematic study of neural network operators (NNOs) can be traced back to the work of Cardaliaguet and Euvrard~\cite{cardaliaguet1992approximation}, who proposed some of the early operator-based neural network models. This line of research was later significantly advanced by Anastassiou and collaborators, who developed a detailed approximation theory for sigmoidal NNOs~\cite{anastassiou1997rate,anastassiou2000rate,anastassiou2011multivariatesigmoidal}, including multivariate generalizations and convergence results in different normed spaces. In recent years, the theory has witnessed significant advancements, including density results, quantitative estimates, and nonlinear extensions of NNOs, which have substantially broadened the scope of the approximation framework~\cite{costarelli2014convergence,costarelli2022density,costarelli2022quantitative, costarelli2024Durrmeyer, berisha2025riemann,saini2026approximation,kadak2025durrmeyer}. These developments have enabled the study of a wider class of activation functions and facilitated more general operator-based constructions.

Despite these developments, most existing NNOs still provide only low-order approximations. This is mainly because their error is governed by the first modulus of continuity, even when the function being approximated is relatively smooth. As a result, their performance becomes limited in applications where high accuracy is essential. To address this challenge, Kadak~\cite{kadak2025hermite} proposed Hermite-type NNOs, which make use of derivative information to achieve higher-order approximation on uniform grids.

Another important line of research focuses on interpolation-based neural network operators, which aim to reconstruct functions directly from discrete data samples. The foundational work of Costarelli~\cite{costarelli2014interpolation-ramp, costarelli2015interpolation-ramp-multi} laid the basis for these interpolation frameworks. Building on this, several extensions have been developed, including fuzzy and multivariate interpolation operators, with their convergence behavior analyzed using moduli of smoothness~\cite{qian2022rates, wang2023Multi, li2019constructive,qian2022smooth-neural,kadak2022multivariate,kadak2022multivariate_fuzzy}.

However, many existing interpolation-based NNO constructions are formulated for uniform sampling grids, which may limit their direct applicability to physical-field reconstruction problems involving nonuniformly distributed measurements. Such settings arise naturally in spatial monitoring applications, including UAV-based environmental mapping, where measurements are collected at discrete locations and the underlying fields may exhibit substantial spatial variability \cite{bakirci2024drone,bakirci2024efficient}. This motivates the development of interpolation operators that accommodate quasi-uniform irregular grids while retaining rigorous approximation guarantees. To improve approximation quality, Lagrange-type neural network interpolation operators on irregular grids were recently introduced by Costarelli et al.~\cite{costarelli2025higher}.

This paper introduces a novel class of Taylor-accelerated neural network interpolation operators (TNNIOs) defined on quasi-uniform irregular grids for higher-order approximation from irregular data. 

\textbf{AI Contribution:}
The proposed framework integrates neural network interpolation operators with local Taylor expansions of degree $r$, enabling explicit utilization of the smoothness properties of the target function. This construction leads to improved approximation accuracy and enhanced convergence rates of order $r+1$, while preserving exact interpolation at the sampling nodes. Theoretical analysis establishes polynomial reproduction, well-posedness, uniform boundedness, and Jackson-type error estimates in terms of higher-order moduli of smoothness, demonstrating clear advantages over classical first-order and Lagrange-type methods.

\textbf{Physical Application:}
The proposed operators are particularly suited for reconstructing derivative-informed physical fields from irregular sensor measurements, a common challenge in modern engineering systems. The framework is applicable to scenarios such as thermal field estimation, electromagnetic signal reconstruction, and fluid flow modeling, including velocity and pressure fields, especially in environments with non-uniform and incomplete data. 

\textbf{Validation and Results:}
The effectiveness of the proposed approach is demonstrated through numerical experiments on representative physical systems. The results show that the method produces smooth, stable, and physically consistent reconstructions, with improved accuracy and robustness compared to existing interpolation techniques. 

This is how the rest of the paper is structured. We start by introducing the nomenclature and basic ideas needed for the analysis, such as higher-order moduli of smoothness, irregular sample grids, compactly supported activation functions, and the functional setup. Following the development of the suggested operators and their theoretical characteristics, applications and numerical validation come next.

Throughout the paper, all functions under consideration are real-valued and defined on the compact interval $[a,b]\subset\mathbb{R}$.

\textbf{Space of continuous functions.}
Let $C[a,b]$ denote the Banach space consisting of all continuous functions
$\chi:[a,b]\to\mathbb{R}$, equipped with the supremum norm
\[
\|\chi\|_\infty := \sup_{t\in[a,b]} |\chi(t)|.
\]
\textbf{Space of smooth functions.}
For a fixed integer $r\ge 1$, we write $C^r[a,b]$ for the space of functions
whose derivatives up to order $r$ exist and are continuous on $[a,b]$, that
is,
\[
\chi^{(j)} \text{ exists and is continuous on } [a,b], \qquad
j=0,1,\ldots,r.
\]
Endowed with the norm
\[
\|\chi\|_{C^r}:=\sum_{j=0}^{r}\|\chi^{(j)}\|_\infty,
\]
the space $C^r[a,b]$ forms a Banach space that is continuously embedded in
$C[a,b]$.

\medskip
Since the TNNIOs developed in this paper rely on local Taylor expansions of degree $r$, the analysis will primarily be carried out for functions $\chi\in C^r[a,b]$. In order to establish quantitative approximation rates, an additional smoothness assumption $\chi\in C^{r+1}[a,b]$ will be imposed when calculating error bounds.

\medskip
\noindent
\textbf{Irregular grids.} Let $t\in[a,b]$ and let $\mathcal{P}_n=\{\xi_k\}_{k=0}^n$ be a partition of $[a,b]$ with $\xi_0=a$ and $\xi_n=b$.
The partition is assumed to be quasi-uniform; i.e., $\exists$ constants
$h_*, h^*>0$ such that \[0<h_* \le \xi_{k+1}-\xi_k \le h^*, \quad k=0,1,\dots,n-1,\]
together with the quasi-uniformity condition
\[h^* < 2h_*.\]

\begin{definition}
A measurable mapping $\eta:\mathbb{R}\to\mathbb{R}$ is said to be \emph{sigmoidal} if it satisfies the asymptotic conditions
\[ \lim_{t\to-\infty}\eta(t)=0 \quad \text{and} \quad
\lim_{t\to+\infty}\eta(t)=1.\]
\end{definition}

In the present study, we focus on a subclass of sigmoidal functions denoted by $D(\ell)$, where $\ell\in\mathbb{R}^+$. A function $\eta:\mathbb{R}\to\mathbb{R}$ belongs to $D(\ell)$ if it satisfies the following properties:
\begin{enumerate} 
\item $\eta$ is increasing; 
\item $\eta(t)=0$ for $t\le -\ell$ and $\eta(t)=1$ for $t\ge \ell$; 
\item $0<\eta(t)<1$ for $-\ell<t<\ell$. 
\end{enumerate}

Associated with $\eta\in D(\ell)$, we define the activation function
$\phi_\eta:\mathbb{R}\to\mathbb{R}$ by \[\phi_\eta(t):=\eta(t+\ell)-\eta(t-\ell).\]
The activation function $\phi_\eta$ satisfies the following fundamental
properties:
\begin{itemize}
\item[\textbf{(M1)}] $\phi_\eta(t)\ge 0$ for all $t\in\mathbb{R}$.
\item[\textbf{(M2)}] $\phi_\eta$ is increasing on $(-\infty,0]$ and decreasing on $[0,\infty)$.
\item[\textbf{(M3)}] $\phi_\eta$ has compact support, namely
\[ \phi_\eta(t)=0, \quad |t|\ge 2\ell.\]
\item[\textbf{(M4)}] \(\phi_\eta(t)+\phi_\eta(t-2\ell)=1, \quad t\in[0,2\ell].\)

\item[\textbf{(M5)}] \( \phi_\eta(0)=1, \quad \phi_\eta\!\left(\frac{2\ell}{h_*}(\xi_j-\xi_k)\right)=0, \quad j\neq k.\)

Indeed, since the grid $\{\xi_k\}_{k=0}^n$ is strictly increasing and satisfies $\xi_{k+1}-\xi_k\ge h_*$ for all $k$, it follows that
\[ |\xi_j-\xi_k|\ge h_*, \quad j\neq k.\]
Consequently,
\[\left|\frac{2\ell}{h_*}(\xi_j-\xi_k)\right| \ge
\frac{2\ell}{h_*}\,|\xi_j-\xi_k| \ge 2\ell.\]
Therefore, by the compact support property \textbf{(M3)}, we obtain
\[\phi_\eta\!\left(\frac{2\ell}{h_*}(\xi_j-\xi_k)\right)=0, \quad j\neq k.\]
\end{itemize}
For further details on sigmoidal functions of class $D(m)$ and their
properties, we refer the reader to~\cite{costarelli2025higher}.

\begin{definition}\label{def:smooth_modulus}
Let $\chi\in C[a,b]$ and let $\delta>0$. The $(r+1)$-st modulus of smoothness
of $\chi$ over the interval $[a,b]$ is defined as
\[\omega_{r+1}(\chi,\delta) := \sup_{|h|\le \delta} \big\|\Delta_h^{\,r+1}\chi\big\|_\infty,\]
where $\Delta_h^{\,r+1}$ denotes the forward finite difference operator of
order $r+1$, given explicitly by
\[\Delta_h^{\,r+1}\chi(t) := \sum_{j=0}^{r+1} (-1)^{\,r+1-j}\binom{r+1}{j}\,
\chi(t+jh),\]
provided that all translated points $t+jh$ remain inside the interval $[a,b]$.
\end{definition}
It is a standard result that additional smoothness of the function $\chi$ yields an upper bound for the modulus of smoothness. More precisely, if $\chi\in C^{r+1}[a,b]$, then
\[\omega_{r+1}(\chi,\delta) \le C_r\,\delta^{r+1}\|\chi^{(r+1)}\|_\infty,\]
where the constant $C_r>0$ depends only on the integer $r$ (see, for instance,
\cite{devore1993constructive, timan2014theory, ditzian2012moduli}).

\begin{lemma}\label{lem:taylor_modulus}
Let $r\in\mathbb{N}_0$ and assume $\chi\in C^{r+1}[a,b]$. Then $\exists$  $C_r>0$, depending solely on $r$, such that for any $t,s\in[a,b]$ the estimate
\[|\chi(t)-P_r(\chi;t,s)| \le C_r\,\omega_{r+1}(\chi,|t-s|)\] holds, where $P_r(\chi;t,s)$ denotes the Taylor polynomial of $\chi$ of degree $r$ centered at the point $s$.
\end{lemma}

\begin{remark}
The estimate stated in Lemma~\ref{lem:taylor_modulus} is a well-established consequence of the theory of finite differences and higher-order moduli of smoothness in approximation theory; see, for example,~\cite{devore1993constructive, timan2014theory, ditzian2012moduli}.
\end{remark}

\section{Taylor-Accelerated Neural Network Interpolation Operator on Irregular Grids}
In this section, we present a TNNIOs defined on irregular sampling grids. The proposed construction enriches standard neural network interpolation schemes by embedding local higher-order Taylor information at the grid nodes, while maintaining exact interpolation at the sampling points.

Let $\chi\in C^{r}[a,b]$, where $r\ge 1$ is a fixed integer. Let $\mathcal{P}_n=\{\xi_k\}_{k=0}^n$ be a quasi-uniform partition of $[a,b]$.
For each grid node $\xi_k$, we consider the Taylor polynomial of degree $r$ centred at $\xi_k$, defined by
\[ P_r(\chi;t,\xi_k) := \sum_{j=0}^{r} \frac{\chi^{(j)}(\xi_k)}{j!}(t-\xi_k)^j,
\quad t\in[a,b].\]
\begin{definition}
Let $\chi\in C^{r}[a,b]$. The Taylor-accelerated neural network interpolation
operator on an irregular grid node $\mathcal{P}_n$ is defined by
\[ \widetilde{T}_{n,r}(\chi,t)= \frac{ \displaystyle\sum_{k=0}^{n} \phi_\eta\!\left(\frac{2\ell}{h_*}(t-\xi_k)\right)
P_r(\chi;t,\xi_k)}{\displaystyle\sum_{k=0}^{n}
\phi_\eta\!\left(\frac{2\ell}{h_*}(t-\xi_k)\right)
}, \quad t\in[a,b], \] where $\phi_\eta$ is the compactly supported activation function introduced in Section~1 and $h_*$ denotes the minimum mesh size of the grid
$\mathcal{P}_n$.
\end{definition}

Now we establish the well-definedness and uniform boundedness of the TNNIOs $\widetilde{T}_{n,r}(\cdot)$.

\begin{proposition}\label{prop:well_defined}
For every $\chi\in C^{r}[a,b]$ and every $t\in[a,b]$, the operator $\widetilde{T}_{n,r}(\chi,t)$ is well defined.
\end{proposition}

\begin{proof}
Consider the denominator
\[ D(t):= \sum_{k=0}^{n} \phi_\eta\!\left(\frac{2\ell}{h_*}(t-\xi_k)\right),
\quad t\in[a,b].\]
By the compact support property \textbf{(M3)} of $\phi_\eta$, for each fixed
$t\in[a,b]$, only finitely many terms in the above sum are nonzero.

Fix $t\in[a,b]$. Since $\{\xi_k\}_{k=0}^n$ is a partition of $[a,b]$, there
exists an index $i\in\{0,\dots,n-1\}$ such that $t\in[\xi_i,\xi_{i+1}]$.
Consequently,
\[\min\{|t-\xi_i|,|t-\xi_{i+1}|\} \le \frac{\xi_{i+1}-\xi_i}{2} \le \frac{h^*}{2}.\]
Let $k_0\in\{i,i+1\}$ be such that \[|t-\xi_{k_0}|\le \frac{h^*}{2}.\]
Using the mesh regularity assumption $h^*<2h_*$, we obtain
\[\left|\frac{2\ell}{h_*}(t-\xi_{k_0})\right| \le \frac{\ell h^*}{h_*} < 2\ell \implies \frac{2\ell}{h_*}(t-\xi_{k_0})\in(-2\ell,2\ell).\]
Since $\eta$ is increasing, it follows that
$\eta(t+\ell)>\eta(t-\ell)$ for all $t\in(-2\ell,2\ell)$, and therefore
\[ \phi_\eta(t)=\eta(t+\ell)-\eta(t-\ell)>0, \quad t\in(-2\ell,2\ell).\]
In particular,
\[\phi_\eta\!\left(\frac{2\ell}{h_*}(t-\xi_{k_0})\right)>0.\]
Consequently,
\[D(t) \ge \phi_\eta\!\left(\frac{2\ell}{h_*}(t-\xi_{k_0})\right)
>0, \quad t\in[a,b].\]
Thus, the denominator in the definition of $\widetilde{T}_{n,r}(\chi,t)$ is strictly positive for all $t\in[a,b]$, and the operator $\widetilde{T}_{n,r}(\chi,t)$ is well defined.
\end{proof}

\begin{theorem}
The operator $\widetilde{T}_{n,r}$ is uniformly bounded on $C^{r}[a,b]$.
More precisely, $\exists$ $C_r>0$, depending only on $r$,
the activation function $\phi_\eta$, and the mesh parameters $h_*,h^*$, such that
\[\|\widetilde{T}_{n,r}(\chi)\|_\infty \le C_r\,\|\chi\|_{C^{r}}, \quad \chi\in C^{r}[a,b].\]
\end{theorem}

\begin{proof}
Let $\chi\in C^{r}[a,b]$ and fix $t\in[a,b]$. By the definition of $\widetilde{T}_{n,r}$ and the non-negativity of $\phi_\eta$, we have
\begin{equation}\label{eq:basic_bound}
|\widetilde{T}_{n,r}(\chi,t)| \le
\frac{\sum\limits_{k=0}^{n}\phi_\eta\!\left(\frac{2\ell}{h_*}(t-\xi_k)\right)
\,|P_r(\chi;t,\xi_k)|}{\sum\limits_{k=0}^{n}\phi_\eta\!\left(\frac{2\ell}{h_*}(t-\xi_k)\right)}.
\end{equation}
Recalling that
\[P_r(\chi;t,\xi_k) =\sum_{j=0}^{r} \frac{\chi^{(j)}(\xi_k)}{j!}(t-\xi_k)^j,\]
we obtain
\begin{equation}\label{eq:taylor_est}
|P_r(\chi;t,\xi_k)| \le \sum_{j=0}^{r} \frac{\|\chi^{(j)}\|_\infty}{j!}\,|t-\xi_k|^{j}.
\end{equation}
By the compact support property \textbf{(M3)} of $\phi_\eta$, if
\[\phi_\eta\!\left(\frac{2\ell}{h_*}(t-\xi_k)\right)\neq 0,\]
then
\[\left|\frac{2\ell}{h_*}(t-\xi_k)\right|<2\ell \implies
|t-\xi_k|<h_*.\]
Therefore, for all $j=0,1,\dots,r$, \[|t-\xi_k|^j \le h_*^{\,j}.\]
Combining this estimate with \eqref{eq:taylor_est}, we obtain
\[|P_r(\chi;t,\xi_k)| \le \left(\sum_{j=0}^{r}\frac{h_*^{\,j}}{j!}\right)
\|\chi\|_{C^{r}} =: M_r\,\|\chi\|_{C^{r}}.\]
Substituting into \eqref{eq:basic_bound} yields
\[|\widetilde{T}_{n,r}(\chi,t)| \le \frac{M_r}{\sum\limits_{k=0}^n
\phi_\eta\!\left(\frac{2\ell}{h_*}(t-\xi_k)\right)}
\|\chi\|_{C^{r}}.\]
By Proposition~\ref{prop:well_defined}, for each $t\in[a,b],$ $\exists$
an index $k_0\in\{0,\dots,n\}$ such that
\[\left|\frac{2\ell}{h_*}(t-\xi_{k_0})\right| \le \frac{\ell h^*}{h_*} < 2\ell. \]
Since $\phi_\eta$ is continuous and strictly positive on $(-2\ell,2\ell)$, we may
define
\[c_{\phi_\eta} := \min_{|x|\le \frac{\ell h^*}{h_*}} \phi_\eta(x)
\;>\;0.\]
Therefore,
\[ \sum_{k=0}^n \phi_\eta\!\left(\frac{2\ell}{h_*}(t-\xi_k)\right) \ge
\phi_\eta\!\left(\frac{2\ell}{h_*}(t-\xi_{k_0})\right)
\ge c_{\phi_\eta}, \quad t\in[a,b].\]
Consequently,
\[|\widetilde{T}_{n,r}(\chi,t)| \le \frac{M_r}{c_{\phi_\eta}}\,\|\chi\|_{C^{r}},
\quad t\in[a,b].\]
Taking the supremum over $t\in[a,b]$ completes the proof.
\end{proof}

\begin{remark}
The uniform boundedness of $\widetilde{T}_{n,r}$ guarantees the stability of the TNNIOs. In particular, small perturbations of $\chi$ in the $C^{r}$-norm lead to controlled perturbations in the approximation $\widetilde{T}_{n,r}(\chi)$.
\end{remark}

\subsection{Approximation Estimate}
In this subsection, we study the approximation properties of the TNNIOs $\widetilde{T}_{n,r}$. In particular, we establish its polynomial reproduction and exact interpolation properties, and derive Jackson-type approximation estimates in terms of the modulus of smoothness. Several corollaries are also presented to illustrate the convergence behaviour under additional regularity assumptions.

\begin{proposition}[Polynomial Reproduction]
Let $r\ge 1$ and let $p$ be a polynomial of degree at most $r$. Then the TNNIOs reproduces $p$ exactly, namely,
\[\widetilde{T}_{n,r}(p,t)=p(t), \quad \forall\,t\in[a,b].\]
\end{proposition}

\begin{proof}
Since $\deg(p)\le r$, we have $p^{(r+1)}\equiv 0$. Therefore, the Taylor polynomial of degree $r$ of $p$ about any grid node$\xi_k$ coincides with $p$ itself, i.e.,  for all $t\in[a,b]$ and all $k=0,1,\dots,n$, \[P_r(p;t,\xi_k)=p(t).\]
Substituting this identity into the definition of
$\widetilde{T}_{n,r}$, we obtain
\[\widetilde{T}_{n,r}(p,t) =\frac{\displaystyle\sum_{k=0}^{n}\phi_\eta\!\left(\frac{2\ell}{h_*}(t-\xi_k)\right)p(t)}{\displaystyle\sum_{k=0}^{n}\phi_\eta\!\left(\frac{2\ell}{h_*}(t-\xi_k)\right)}.\]
Factoring out $p(t)$ yields
\[\widetilde{T}_{n,r}(p,t)= p(t)\, \frac{\displaystyle\sum_{k=0}^{n}\phi_\eta\!\left(\frac{2\ell}{h_*}(t-\xi_k)\right)
}{\displaystyle\sum_{k=0}^{n}\phi_\eta\!\left(\frac{2\ell}{h_*}(t-\xi_k)\right)}
= p(t).\]
This proves the polynomial reproduction property.
\end{proof}

\begin{theorem}[Interpolation Property]
Let $\chi\in C^{r}[a,b]$. Then, for every grid node $\xi_j$, $j=0,1,\dots,n$, the TNNIOs satisfies \[\widetilde{T}_{n,r}(\chi,\xi_j)=\chi(\xi_j).\]
\end{theorem}

\begin{proof}
Fix a grid node $\xi_j$. By definition,
\[\widetilde{T}_{n,r}(\chi,\xi_j)=\frac{\displaystyle\sum_{k=0}^{n} \phi_\eta\!\left(\frac{2\ell}{h_*}(\xi_j-\xi_k)\right) P_r(\chi;\xi_j,\xi_k)}{\displaystyle\sum_{k=0}^{n} \phi_\eta\!\left(\frac{2\ell}{h_*}(\xi_j-\xi_k)\right)}.\]
First, the Taylor polynomial of degree $r$ centered at $\xi_j$ satisfies
\[ P_r(\chi;\xi_j,\xi_j) = \sum_{s=0}^{r}\frac{\chi^{(s)}(\xi_j)}{s!}(\xi_j-\xi_j)^s
= \chi(\xi_j).\]
Next, if $k\neq j$, then since the grid $\{\xi_k\}_{k=0}^n$ is strictly increasing and satisfies $\xi_{k+1}-\xi_k\ge h_*$ for all $k$, we have
\[|\xi_j-\xi_k|\ge h_*. \]
Consequently,
\[\left|\frac{2\ell}{h_*}(\xi_j-\xi_k)\right|\ge 2\ell.\]
Therefore, by the compact support property \textbf{(M3)},
\[\phi_\eta\!\left(\frac{2\ell}{h_*}(\xi_j-\xi_k)\right)=0, \quad k\neq j.\]
On the other hand,
\[\phi_\eta\!\left(\frac{2\ell}{h_*}(\xi_j-\xi_j)\right)=\phi_\eta(0)=1.\]
Hence, only the term $k=j$ contributes in both the numerator and the denominator, and we obtain
\[\widetilde{T}_{n,r}(\chi,\xi_j)= \frac{\phi_\eta(0)\,P_r(\chi;\xi_j,\xi_j)}{\phi_\eta(0)}
= \chi(\xi_j).\]
This proves the interpolation property.
\end{proof}

\begin{theorem}[Jackson-type estimate]\label{thm:jackson}
Let $\chi\in C^{r+1}[a,b]$. Then $\exists$ $C_r>0$, depending only on $r$, the activation function $\phi_\eta$, and the mesh parameters, such that
\[\|\widetilde{T}_{n,r}(\chi)-\chi\|_\infty \le C_r\,\omega_{r+1}(\chi,h_n),\]
where \[h_n:=\max_{0\le k\le n-1}(\xi_{k+1}-\xi_k)\]
denotes the mesh norm of the irregular grid.
In particular, if $h_n\to0$, then $\widetilde{T}_{n,r}(\chi)\to \chi$ uniformly on $[a,b]$.
\end{theorem}

\begin{proof}
Fix an arbitrary point $t\in[a,b]$. By definition of the TNNIOs, we have
\[ \widetilde{T}_{n,r}(\chi,t)-\chi(t) = \frac{
\sum\limits_{k=0}^{n} \phi_\eta\!\left(\frac{2\ell}{h_*}(t-\xi_k)\right)
\big(P_r(\chi;t,\xi_k)-\chi(t)\big)}{\sum\limits_{k=0}^{n}\phi_\eta\!\left(\frac{2\ell}{h_*}(t-\xi_k)\right)}.\]
Taking absolute values and using the non-negativity property \textbf{(M1)} of $\phi_\eta$, we obtain
\begin{equation}\label{eq:jackson_modulus}
|\widetilde{T}_{n,r}(\chi,t)-\chi(t)| \le \frac{ \sum\limits_{k=0}^{n} \phi_\eta\!\left(\frac{2\ell}{h_*}(t-\xi_k)\right) \,|P_r(\chi;t,\xi_k)-\chi(t)|}{\sum\limits_{k=0}^{n} \phi_\eta\!\left(\frac{2\ell}{h_*}(t-\xi_k)\right)}.
\end{equation}
By the compact support property \textbf{(M3)} of $\phi_\eta$, the summand in
\eqref{eq:jackson_modulus} are nonzero only for those indices $k$ such that
\[\left|\frac{2\ell}{h_*}(t-\xi_k)\right|<2\ell,\]
which implies \[|t-\xi_k|<h_*\le h_n.\]
By Lemma~\ref{lem:taylor_modulus}, for each index $k$ such that
$\phi_\eta\!\left(\frac{2\ell}{h_*}(t-\xi_k)\right)\neq 0$, we have
\[|P_r(\chi;t,\xi_k)-\chi(t)| \le C_r\,\omega_{r+1}(\chi,|t-\xi_k|).\]
Since the modulus of smoothness is non-decreasing in its second argument and
$|t-\xi_k|\le h_n$, it follows that
\[|P_r(\chi;t,\xi_k)-\chi(t)| \le C_r\,\omega_{r+1}(\chi,h_n).\]
Substituting this estimate into \eqref{eq:jackson_modulus} yields
\[ |\widetilde{T}_{n,r}(\chi,t)-\chi(t)| \le C_r\,\omega_{r+1}(\chi,h_n), \quad t\in[a,b].\]
Taking the supremum over $t\in[a,b]$ completes the proof.
\end{proof}

\begin{corollary}\label{cor:smooth-rate}
Let $\chi\in C^{r+1}[a,b]$ and let \(h_n:=\max\limits_{0\le k\le n-1}(\xi_{k+1}-\xi_k)\)
denote the mesh norm of the irregular grids. Then the TNNIOs
$\widetilde{T}_{n,r}$ satisfies
\[\|\widetilde{T}_{n,r}(\chi)-\chi\|_\infty \le \widehat{C}_r\,h_n^{\,r+1}\,\|\chi^{(r+1)}\|_\infty,\]
where $\widehat{C}_r>0$ is a constant independent of $n$, depending only on $r$.
\end{corollary}

\begin{proof}
By Theorem~\ref{thm:jackson}, $\exists$ $A_r>0$, depending only on $r$, such that
\[\|\widetilde{T}_{n,r}(\chi)-\chi\|_\infty \le A_r\,\omega_{r+1}(\chi,h_n).\]
Since $\chi\in C^{r+1}[a,b]$, the inequality from Definition~\ref{def:smooth_modulus} for higher-order moduli of smoothness yield
\[ \omega_{r+1}(\chi,\delta) \le B_r\,\delta^{\,r+1}\,\|\chi^{(r+1)}\|_\infty,
\quad \delta>0,\] where $B_r>0$ depends only on $r$.

Applying this estimate with $\delta=h_n$, we obtain
\[\omega_{r+1}(\chi,h_n) \le B_r\,h_n^{\,r+1}\,\|\chi^{(r+1)}\|_\infty.\]
Combining the above inequalities, we conclude that
\[\|\widetilde{T}_{n,r}(\chi)-\chi\|_\infty \le (A_rB_r)\,h_n^{\,r+1}\,\|\chi^{(r+1)}\|_\infty.\]
Setting $\widehat{C}_r:=A_rB_r$ completes the proof.
\end{proof}

\begin{corollary}\label{cor:n-rate}
Assume that the sequence of irregular grids satisfies the refinement condition,  i.e., $\exists$ $C>0$, independent of $n$, such that \[h_n:=\max_{0\le k\le n-1}(\xi_{k+1}-\xi_k) \le \frac{C}{n}.\]
Such a condition is fulfilled, for instance, by any quasi-uniform family
of grids of $[a,b]$.

Let $\chi\in C^{r+1}[a,b]$. Then the TNNIOs $\widetilde{T}_{n,r}$ satisfies
\[ \|\widetilde{T}_{n,r}(\chi)-\chi\|_\infty \le \widetilde{C}_r\,n^{-(r+1)}\,\|\chi^{(r+1)}\|_\infty,\]
where $\widetilde{C}_r>0$ is independent of $n$ and depends only on $r$ and the operator parameters.
\end{corollary}

\begin{proof}
By Corollary~\ref{cor:smooth-rate}, we have
\[ \|\widetilde{T}_{n,r}(\chi)-\chi\|_\infty \le \widehat{C}_r\,h_n^{\,r+1}\,\|\chi^{(r+1)}\|_\infty.\]
Using the refinement assumption $h_n\le \frac{C}{n}$, it follows that
\[ h_n^{\,r+1} \le \left(\frac{C}{n}\right)^{r+1} = C^{r+1}\,n^{-(r+1)}.\]
Substituting this into the previous inequality yields
\[\|\widetilde{T}_{n,r}(\chi)-\chi\|_\infty \le \widehat{C}_r\,C^{r+1}\, n^{-(r+1)}\,\|\chi^{(r+1)}\|_\infty.\]
Defining \(\widetilde{C}_r:=\widehat{C}_r\,C^{r+1}\)
completes the proof.
\end{proof}

\begin{remark}[Optimal order of convergence]
Corollary~\ref{cor:smooth-rate} shows that, for functions $\chi\in C^{r+1}[a,b]$, the TNNIOs $\widetilde{T}_{n,r}$ achieves the optimal approximation order $O(h_n^{\,r+1})$, which coincides with the best possible rate for polynomial-based interpolation schemes of degree $r$.
\end{remark}

Although the operator introduced in \cite{costarelli2025higher} is defined for all continuous functions $\chi\in C[a,b]$, its higher-order accuracy is obtained through a Lagrange-type reconstruction involving $(r+1)$ neighbouring sampling grids. On irregular grids, this may lead to additional numerical complexity and reduced stability due to the behaviour of Lagrange basis polynomials.

In contrast, the Taylor-accelerated operator proposed in this paper is specifically designed for smooth functions in $C^{r}[a,b]$, where local Taylor expansions can be exploited directly. In the next section, numerical experiments confirm that, for sufficiently regular functions, the proposed method produces a significantly faster decay of the approximation error compared with the Lagrange-based interpolation operator of \cite{costarelli2025higher}.
This highlights the effectiveness of the Taylor acceleration mechanism in practical computations.

\section{Multivariate Extension}
We now extend the above construction to the multivariate setting. Let $d\in\mathbb{N}$ and consider the $d$-dimensional box domain
\[\Omega := [a_1,b_1]\times\cdots\times[a_d,b_d]\subset\mathbb{R}^d.\]
Let $\eta\in D(\ell)$ and let $\phi_\eta:\mathbb{R}\to\mathbb{R}$ be the
associated activation function introduced in Section~1. We define the multidimensional activation function \[\Psi_{\phi_\eta}:\mathbb{R}^d\to\mathbb{R}\]
by
\[\Psi_{\phi_\eta}(t) := \prod_{i=1}^{d}\phi_\eta(t_i),
\quad t = (t_1,\dots,t_d)\in\mathbb{R}^d.\]
The function $\Psi_{\phi_\eta}$ is nonnegative, compactly supported, and
satisfies
\[\|\Psi_{\phi_\eta}\|_\infty := \sup_{t\in\mathbb{R}^d}|\Psi_{\phi_\eta}(t)| \le 1.\]

\textbf{Multivariate irregular grids.}
For simplicity of exposition, we assume that the same number $n$ of grid
points are used in each coordinate direction.
For each $i\in\{1,\dots,d\}$, let
\[a_i=\xi_{i0}<\xi_{i1}<\cdots<\xi_{in}=b_i\]
be an irregular partition of $[a_i,b_i]$.
We assume that $\exists$ $0<h_*\le h^*$, with $h^*<2h_*$, such
that
\[ h_*\le \xi_{i,k+1}-\xi_{i,k}\le h^*, \quad
k=0,1,\dots,n-1,\ \ i=1,\dots,d. \]
For a multi-index $k=(k_1,\dots,k_d)\in\{0,\dots,n\}^d$, we define the
corresponding grid node \[\xi_k := (\xi_{1k_1},\dots,\xi_{dk_d})\in\Omega.\]

\textbf{Multivariate Taylor polynomial.}
Let $\chi\in C^r(\Omega)$. For each grid node $\xi_k$, we denote by
$P_r(\chi; t,\xi_k)$ the multivariate Taylor polynomial of total degree $r$
of $\chi$ centered at $\xi_k$, defined by
\[ P_r(\chi;t,\xi_k) = \sum_{|\alpha|\le r} \frac{D^\alpha \chi(\xi_k)}{\alpha!} (t-\xi_k)^\alpha, \quad t\in\Omega, \]
where $\alpha=(\alpha_1,\dots,\alpha_d)\in\mathbb{N}_0^d$ is a multi-index,
$|\alpha|=\alpha_1+\cdots+\alpha_d$, $\alpha!=\alpha_1!\cdots\alpha_d!$, and
\[ D^\alpha \chi := \frac{\partial^{|\alpha|} \chi} {\partial t_1^{\alpha_1}\cdots\partial t_d^{\alpha_d}}.\]

\begin{definition}\label{def:multi_smooth_modulus}
Let $\chi\in C(\Omega)$ and $\delta>0$.  
The $(r+1)$-th modulus of smoothness of $\chi$ on $\Omega$ is defined by
\[\omega_{r+1}(\chi,\delta) := \sup_{\|h\|_\infty\le \delta}
\big\|\Delta_h^{\,r+1} \chi\big\|_\infty,\]
where $\|h\|_\infty=\max\limits_{1\le i\le d}|h_i|$ and the forward difference
operator of order $r+1$ is given by
\[\Delta_h^{\,r+1} \chi(t) := \sum_{j=0}^{r+1}
(-1)^{r+1-j}\binom{r+1}{j}\,\chi(t+jh),\]
whenever all shifted points $t+jh$ belong to $\Omega$.
\end{definition}

\begin{lemma}\label{lem:multi_taylor_modulus}
Let $r\in\mathbb{N}_0$ and let $\chi\in C^{r+1}(\Omega)$. Then $\exists$
$C_{r,d}>0$, depending only on $r$ and $d$, such that for all $t,s\in\Omega$,
\[ |\chi(t)-P_r(\chi;t,s)| \le C_{r,d}\,\omega_{r+1}(\chi,\|t-s\|_\infty),\]
where $P_r(\chi;t,s)$ denotes the multivariate Taylor polynomial of total degree
$r$ of $\chi$ centered at $s$.
\end{lemma}

\begin{remark}
Lemma~\ref{lem:multi_taylor_modulus} is a consequence of the theory of multivariate moduli of smoothness and finite differences; see, e.g., \cite{devore1993constructive, timan2014theory, ditzian2012moduli}.
\end{remark}

\begin{definition}
Let $\chi\in C^r(\Omega)$. The multivariate Taylor-accelerated neural network interpolation operator is defined by
\[ \widetilde{T}_{n,r}^{(d)}(\chi,t) = \frac{ \displaystyle \sum_{k\in\{0,\dots,n\}^d}
\Psi_{\phi_\eta}\!\left(\frac{2\ell}{h_*}(t-\xi_k)\right)
P_r(\chi;t,\xi_k)}{ \displaystyle \sum_{k\in\{0,\dots,n\}^d}
\Psi_{\phi_\eta}\!\left(\frac{2\ell}{h_*}(t-\xi_k)\right)
}, \quad t\in\Omega, \]
where the scaling $\frac{2\ell}{h_*}(t-\xi_k)$ is understood component-wise.
\end{definition}

Now, first we will show the well-definedness and boundedness of the operator $\widetilde{T}_{n,r}^{(d)}(\cdot) $.

\begin{proposition}[Well-definedness]
For every $\chi\in C^{r}(\Omega)$ and every $t\in\Omega$, the multivariate TNNIOs
$\widetilde{T}_{n,r}^{(d)}(\chi,t)$ is well defined.
\end{proposition}

\begin{proof}
Consider the denominator
\[ D^{(d)}(t) := \sum_{k\in\{0,\dots,n\}^d} \Psi_{\phi_\eta}\!\left(\frac{2\ell}{h_*}(t-\xi_k)\right), \quad t\in\Omega. \]
Fix an arbitrary point $t=(t_1,\dots,t_d)\in\Omega$.
For each coordinate direction $i\in\{1,\dots,d\}$, since
$\{\xi_{ik}\}_{k=0}^n$ is a partition of $[a_i,b_i]$, $\exists$
$j_i\in\{0,\dots,n-1\}$ such that \[t_i\in[\xi_{i j_i},\xi_{i,j_i+1}].\]
Consequently,
\[ \min\{|t_i-\xi_{i j_i}|,\;|t_i-\xi_{i,j_i+1}|\}
\le \frac{\xi_{i,j_i+1}-\xi_{i j_i}}{2}
\le \frac{h^*}{2}.\]
Let $k_0=(k_{01},\dots,k_{0d})\in\{0,\dots,n\}^d$ be defined by choosing
$k_{0i}\in\{j_i,j_i+1\}$ such that
\[ |t_i-\xi_{i k_{0i}}|\le \frac{h^*}{2},
\quad i=1,\dots,d.\]
Then
\[|t_i-\xi_{i k_{0i}}|\le \frac{h^*}{2}
< h_*, \quad i=1,\dots,d, \]
where we used the assumption $h^*<2h_*$.

Hence,
\[\left|\frac{2\ell}{h_*}(t_i-\xi_{i k_{0i}})\right|<2\ell,
\quad i=1,\dots,d, \]
and therefore
\[\Psi_{\phi_\eta}\!\left(\frac{2\ell}{h_*}(t-\xi_{k_0})\right) = \prod_{i=1}^d
\phi_\eta\!\left(\frac{2\ell}{h_*}(t_i-\xi_{i k_{0i}})\right) >0. \]
Consequently,
\[D^{(d)}(t) \ge \Psi_{\phi_\eta}\!\left(\frac{2\ell}{h_*}(t-\xi_{k_0})\right) >0,
\quad t\in\Omega.\]
Thus, the denominator in the definition of $\widetilde{T}_{n,r}^{(d)}(\chi,t)$ is strictly positive for all $t\in\Omega$, and the operator is well defined.
\end{proof}

\begin{theorem}[Uniform boundedness]
The operator $\widetilde{T}_{n,r}^{(d)}$ is uniformly bounded on
$C^{r}(\Omega)$. More precisely, $\exists$ $C_{r,d}>0$, depending only on $r$, $d$, the activation function $\phi_\eta$, and the mesh parameters $h_*,h^*$, such that
\[\|\widetilde{T}_{n,r}^{(d)}(\chi)\|_\infty \le C_{r,d}\,\|\chi\|_{C^{r}(\Omega)}, \quad \chi\in C^{r}(\Omega).\]
\end{theorem}

\begin{proof}
Let $\chi\in C^{r}(\Omega)$ and fix $t\in\Omega$.
By the definition of $\widetilde{T}_{n,r}^{(d)}$ and the non-negativity of
$\Psi_{\phi_\eta}$, we have
\begin{equation}\label{eq:multi_basic_bound}
|\widetilde{T}_{n,r}^{(d)}(\chi,t)| \le \frac{\sum\limits_{k\in\{0,\dots,n\}^d}\Psi_{\phi_\eta}\!\left(\frac{2\ell}{h_*}(t-\xi_k)\right)
\,|P_r(\chi;t,\xi_k)|}{\sum\limits_{k\in\{0,\dots,n\}^d}
\Psi_{\phi_\eta}\!\left(\frac{2\ell}{h_*}(t-\xi_k)\right)}.
\end{equation}
Recall that the multivariate Taylor polynomial satisfies
\[ P_r(\chi;t,\xi_k) = \sum_{|\alpha|\le r} \frac{D^\alpha \chi(\xi_k)}{\alpha!}(t-\xi_k)^\alpha.\]
Hence,
\[|P_r(\chi;t,\xi_k)| \le \sum_{|\alpha|\le r} \frac{\|D^\alpha \chi\|_\infty}{\alpha!}
|t-\xi_k|^{|\alpha|}.\]
By the compact support of $\Psi_{\phi_\eta}$, if
\[\Psi_{\phi_\eta}\!\left(\frac{2\ell}{h_*}(t-\xi_k)\right)\neq 0,\]
then
\[|t_i-\xi_{i k_i}|<h_*, \quad i=1,\dots,d
\implies |t-\xi_k|^{|\alpha|}
\le (h_*)^{|\alpha|}.\]
It follows that
\[ |P_r(\chi;t,\xi_k)| \le \left( \sum_{|\alpha|\le r}\frac{h_*^{|\alpha|}}{\alpha!}
\right) \|\chi\|_{C^{r}(\Omega)} =: M_{r,d}\,\|\chi\|_{C^{r}(\Omega)}.\]
Substituting this estimate into \eqref{eq:multi_basic_bound}, we obtain
\[|\widetilde{T}_{n,r}^{(d)}(\chi,t)| \le \frac{M_{r,d}}{ \sum\limits_{k\in\{0,\dots,n\}^d}
\Psi_{\phi_\eta}\!\left(\frac{2\ell}{h_*}(t-\xi_k)\right)}
\|\chi\|_{C^{r}(\Omega)}.\]
By the well-definedness result above, there exists a constant
\[ c_{\Psi} := \min_{\|x\|_\infty\le \frac{\ell h^*}{h_*}}
\Psi_{\phi_\eta}(x) >0 \]
such that
\[\sum_{k\in\{0,\dots,n\}^d} \Psi_{\phi_\eta}\!\left(\frac{2\ell}{h_*}(t-\xi_k)\right)
\ge c_{\Psi}, \quad t\in\Omega.\]
Therefore,
\[|\widetilde{T}_{n,r}^{(d)}(\chi,t)| \le \frac{M_{r,d}}{c_{\Psi}}\, \|\chi\|_{C^{r}(\Omega)}, \quad t\in\Omega. \]
Taking the supremum over $t\in\Omega$ completes the proof.
\end{proof}

\subsection{Approximation Estimate}

In this subsection, we study the approximation properties of the TNNIOs $\widetilde{T}_{n,r}^{(d)}$. In particular, we establish its polynomial reproduction and exact interpolation properties, and derive Jackson-type approximation estimates in terms of the modulus of smoothness. Several corollaries are also presented to illustrate the
convergence behaviour under additional regularity assumptions.

\begin{theorem}[Multivariate polynomial reproduction]
Let $r\ge 1$ and let $p$ be a polynomial on $\mathbb{R}^d$ of total degree
at most $r$. Then
\[\widetilde{T}_{n,r}^{(d)}(p,t)=p(t), \quad \forall\,y\in\Omega.\]
\end{theorem}

\begin{proof}
Let $p$ be a polynomial of total degree at most $r$.
Then all partial derivatives of order $r+1$ vanish identically, that is,
\[ D^\alpha p \equiv 0, \quad |\alpha|=r+1.\]
Therefore, the multivariate Taylor polynomial of total degree $r$ of $p$ about any grid node$\xi_k$ coincides exactly with $p$, i.e.,
\[P_r(p;t,\xi_k)=p(t), \quad t\in\Omega,\ \ k\in\{0,\dots,n\}^d.\]
Substituting this identity into the definition of $\widetilde{T}_{n,r}^{(d)}$, we obtain
\[\widetilde{T}_{n,r}^{(d)}(p,t)= \frac{ \sum\limits_{k\in\{0,\dots,n\}^d}
\Psi_{\phi_\eta}\!\left(\frac{2\ell}{h_*}(t-\xi_k)\right)p(t)
}{\sum\limits_{k\in\{0,\dots,n\}^d} \Psi_{\phi_\eta}\!\left(\frac{2\ell}{h_*}(t-\xi_k)\right)}.\]
Since $p(t)$ is independent of $k$, it can be factored out, yielding
\[\widetilde{T}_{n,r}^{(d)}(p,t) = p(t)\, \frac{ \sum\limits_{k\in\{0,\dots,n\}^d}\Psi_{\phi_\eta}\!\left(\frac{2\ell}{h_*}(t-\xi_k)\right)}{ \sum\limits_{k\in\{0,\dots,n\}^d}\Psi_{\phi_\eta}\!\left(\frac{2\ell}{h_*}(t-\xi_k)\right)}
= p(t).\]
This proves the polynomial reproduction property.
\end{proof}

\begin{theorem}[Multivariate interpolation property]
Let $\chi\in C^{r}(\Omega)$. Then, for every grid node $\xi_j\in\Omega$, $j\in\{0,\dots,n\}^d$, the multivariate TNNIOs satisfies \[\widetilde{T}_{n,r}^{(d)}(\chi,\xi_j)=\chi(\xi_j).\]
\end{theorem}

\begin{proof}
Fix a grid node $\xi_j=(\xi_{1j_1},\dots,\xi_{dj_d})\in\Omega$.
By definition,
\[
\widetilde{T}_{n,r}^{(d)}(\chi,\xi_j) = \frac{\sum\limits_{k\in\{0,\dots,n\}^d}
\Psi_{\phi_\eta}\!\left(\frac{2\ell}{h_*}(\xi_j-\xi_k)\right)
P_r(\chi;\xi_j,\xi_k)}{\sum\limits_{k\in\{0,\dots,n\}^d}
\Psi_{\phi_\eta}\!\left(\frac{2\ell}{h_*}(\xi_j-\xi_k)\right)}.
\]
First, the multivariate Taylor polynomial centered at $\xi_j$ satisfies
\[ P_r(\chi;\xi_j,\xi_j) = \sum_{|\alpha|\le r} \frac{D^\alpha \chi(\xi_j)}{\alpha!}
(\xi_j-\xi_j)^\alpha = \chi(\xi_j).\]
Next, consider any index $k\neq j$.
Then there exists at least one coordinate $i\in\{1,\dots,d\}$ such that
$k_i\neq j_i$, and hence
\[|\xi_{ij_i}-\xi_{ik_i}|\ge h_*.\]
Consequently,
\[\left|\frac{2\ell}{h_*}(\xi_{ij_i}-\xi_{ik_i})\right|\ge 2\ell \implies
\phi_\eta\!\left(\frac{2\ell}{h_*}(\xi_{ij_i}-\xi_{ik_i})\right)=0.\]
Therefore,
\[\Psi_{\phi_\eta}\!\left(\frac{2\ell}{h_*}(\xi_j-\xi_k)\right) = \prod_{i=1}^d \phi_\eta\!\left(\frac{2\ell}{h_*}(\xi_{ij_i}-\xi_{ik_i})\right) =0,
\quad k\neq j.\]
On the other hand,
\[ \Psi_{\phi_\eta}(0) = \prod_{i=1}^d \phi_\eta(0) = 1.\]
Hence, only the term $k=j$ contributes in both the numerator and the denominator, and we obtain
\[ \widetilde{T}_{n,r}^{(d)}(\chi,\xi_j) = \frac{\Psi_{\phi_\eta}(0)\,P_r(\chi;\xi_j,\xi_j)}{\Psi_{\phi_\eta}(0)} = \chi(\xi_j). \]
This proves the interpolation property.
\end{proof}

\begin{theorem}[Multivariate Jackson-type estimate]
\label{thm:multi_jackson}
Let $\chi\in C^{r+1}(\Omega)$. Then $\exists$ $C_{r,d}>0$, depending only
on $r$, $d$, the activation function $\phi_\eta$, and the mesh parameters, such
that
\[\|\widetilde{T}_{n,r}^{(d)}(\chi)-\chi\|_\infty \le C_{r,d}\,\omega_{r+1}(\chi,h_n),\]
where
\[h_n:=\max_{1\le i\le d}\max_{0\le k\le n-1}(\xi_{i,k+1}-\xi_{ik})\]
denotes the mesh norm of the multivariate irregular grids.
In particular, if $h_n\to0$, then $\widetilde{T}_{n,r}^{(d)}(\chi)\to \chi$ uniformly on $\Omega$.
\end{theorem}

\begin{proof}
Fix an arbitrary point $t\in\Omega$. By the definition of the multivariate
TNNIOs,
\[\widetilde{T}_{n,r}^{(d)}(\chi,t)-\chi(t) =
\frac{ \sum\limits_{k\in\{0,\dots,n\}^d} \Psi_{\phi_\eta}\!\left(\frac{2\ell}{h_*}(t-\xi_k)\right) \big(P_r(\chi;t,\xi_k)-\chi(t)\big)}{\sum\limits_{k\in\{0,\dots,n\}^d}\Psi_{\phi_\eta}\!\left(\frac{2\ell}{h_*}(t-\xi_k)\right)}.
\]
Taking absolute values and using the non-negativity of $\Psi_{\phi_\eta}$,
we obtain
\begin{equation}\label{eq:multi_jackson_modulus}
|\widetilde{T}_{n,r}^{(d)}(\chi,t)-\chi(t)| \le
\frac{\sum\limits_{k\in\{0,\dots,n\}^d} \Psi_{\phi_\eta}\!\left(\frac{2\ell}{h_*}(t-\xi_k)\right) \,|P_r(\chi;t,\xi_k)-\chi(t)| }{ \sum\limits_{k\in\{0,\dots,n\}^d} \Psi_{\phi_\eta}\!\left(\frac{2\ell}{h_*}(t-\xi_k)\right)}.
\end{equation}
By the compact support of $\Psi_{\phi_\eta}$, the summand in \eqref{eq:multi_jackson_modulus} are nonzero only for indices
$k=(k_1,\dots,k_d)$ such that
\[|t_i-\xi_{ik_i}|<h_* \le h_n, \quad i=1,\dots,d.\]
Hence,
\[\|t-\xi_k\|_\infty \le h_n.\]
By Lemma~\ref{lem:multi_taylor_modulus}, for each such index $k$ we have
\[|P_r(\chi;t,\xi_k)-\chi(t)|\le
C_{r,d}\,\omega_{r+1}(\chi,\|t-\xi_k\|_\infty).\]
Since the modulus of smoothness is non-decreasing in its second argument, it
follows that
\[|P_r(\chi;t,\xi_k)-\chi(t)| \le C_{r,d}\,\omega_{r+1}(\chi,h_n).\]
Substituting this estimate into \eqref{eq:multi_jackson_modulus}, we obtain
\[|\widetilde{T}_{n,r}^{(d)}(\chi,t)-\chi(t)| \le C_{r,d}\,\omega_{r+1}(\chi,h_n), \quad t\in\Omega.\]
Taking the supremum over $t\in\Omega$ completes the proof.
\end{proof}

\begin{corollary}\label{cor:multi_smooth_rate}
Let $\chi\in C^{r+1}(\Omega)$ and let \[h_n:=\max_{1\le i\le d}\max_{0\le k\le n-1}(\xi_{i,k+1}-\xi_{ik})\] denote the mesh norm of the multivariate irregular grids.
Then the multivariate TNNIOs $\widetilde{T}_{n,r}^{(d)}$ satisfies
\[\|\widetilde{T}_{n,r}^{(d)}(\chi)-\chi\|_\infty \le \widehat{C}_{r,d}\, h_n^{\,r+1} \max_{|\alpha|=r+1}\|D^\alpha \chi\|_\infty,\] where $\widehat{C}_{r,d}>0$ is a constant independent of $n$, depending only on $r$ and $d$.
\end{corollary}

\begin{proof}
By Theorem~\ref{thm:multi_jackson}, $\exists$ $A_{r,d}>0$,
depending only on $r$ and $d$, such that
\[\|\widetilde{T}_{n,r}^{(d)}(\chi)-\chi\|_\infty \le
A_{r,d}\,\omega_{r+1}(\chi,h_n).\]
Since $\chi\in C^{r+1}(\Omega)$, the standard inequality for the $(r+1)$-th modulus of smoothness yields the existence of a constant $B_{r,d}>0$, depending only on $r$ and $d$, such that
\[\omega_{r+1}(\chi,\delta) \le B_{r,d}\,\delta^{\,r+1} \max_{|\alpha|=r+1}\|D^\alpha \chi\|_\infty, \quad \delta>0.\]
Applying this inequality with $\delta=h_n$ and combining the above estimates,
we obtain
\[\|\widetilde{T}_{n,r}^{(d)}(\chi)-\chi\|_\infty \le (A_{r,d}B_{r,d})\, h_n^{\,r+1} \max_{|\alpha|=r+1}\|D^\alpha \chi\|_\infty.\]
Setting $\widehat{C}_{r,d}:=A_{r,d}B_{r,d}$ completes the proof.
\end{proof}

\begin{corollary}\label{cor:multi_n_rate}
Assume that the family of multivariate irregular grids satisfies the refinement condition, i.e.,  $\exists$ $C>0$, independent of $n$, such that
\[h_n\le \frac{C}{n}.\]
Such a condition is fulfilled, for instance, by any quasi-uniform family of tensor-product grids on $\Omega$.

Let $\chi\in C^{r+1}(\Omega)$. Then the multivariate TNNIOs $\widetilde{T}_{n,r}^{(d)}$ satisfies
\[\|\widetilde{T}_{n,r}^{(d)}(\chi)-\chi\|_\infty \le \widetilde{C}_{r,d}\, n^{-(r+1)} \max_{|\alpha|=r+1}\|D^\alpha \chi\|_\infty,\]
where $\widetilde{C}_{r,d}>0$ is independent of $n$ and depends only on $r$, $d$, the activation function, and the grid parameters.
\end{corollary}
\begin{proof}
By Corollary~\ref{cor:multi_smooth_rate}, we have
\[\|\widetilde{T}_{n,r}^{(d)}(\chi)-\chi\|_\infty \le \widehat{C}_{r,d}\, h_n^{\,r+1} \max_{|\alpha|=r+1}\|D^\alpha \chi\|_\infty.\]
Using the refinement assumption $h_n\le \frac{C}{n}$, it follows that
\[h_n^{\,r+1} \le \left(\frac{C}{n}\right)^{r+1} = C^{r+1}\,n^{-(r+1)}.\]
Substituting this bound into the previous inequality yields
\[\|\widetilde{T}_{n,r}^{(d)}(\chi)-\chi\|_\infty \le \widehat{C}_{r,d}\,C^{r+1}\, n^{-(r+1)} \max_{|\alpha|=r+1}\|D^\alpha \chi\|_\infty.\]
Setting $\widetilde{C}_{r,d}:=\widehat{C}_{r,d}\,C^{r+1}$ completes the
proof.
\end{proof}

\section{Numerical Validation}

In this section, we present numerical experiments to validate the theoretical approximation results established in the previous sections.
In particular, we illustrate the higher-order convergence behaviour of the proposed TNNIOs and provide a direct comparison with the Lagrange-based neural network interpolation operator on irregular grids recently introduced by Costarelli~\emph{et al.}~\cite{costarelli2025higher}.

We consider the smooth target function \[f(t)=\sin(2\pi t)+t^3, \quad t\in[0,1],\]
which is infinitely differentiable on $[0,1]$ and therefore satisfies the
regularity assumptions required in our convergence theorems.

\subsection{Activation Function and Grid Construction}

Both interpolation operators are based on a compactly supported sigmoidal activation function. In the numerical experiments, we adopt a truncated ramp-type sigmoid $\eta\in D(m)$ defined by
\[ \eta(y)=
\begin{cases}
0, & t\le -\ell,\\[6pt]
\dfrac{t+\ell}{2\ell}, & -\ell<t<\ell,\\[10pt]
1, & t\ge \ell,
\end{cases}
\]
where $m>0$ is a fixed support parameter. The associated activation function is given by
\[\phi_\eta(t)=\eta(t+\ell)-\eta(t-\ell),\] which satisfies $\phi_\eta(0)=1$ and has compact support contained in $[-2\ell,2\ell]$.

All experiments were implemented in MATLAB. The sampling grids were chosen as nested quasi-uniform irregular grids of $[0,1]$, generated by applying small random perturbations to a uniform grid node while preserving the mesh regularity condition \(\frac{h^*}{h_*}< 2.\)
This construction ensures that the theoretical assumptions are satisfied and allows us to study convergence as the number of grids $n$ increases.

\subsection{Error Convergence and Comparison}

For each grid node size $n$ and Taylor order $r$, we compute the uniform
approximation error \[E^{\mathrm{Tay}}_{n,r}:= \bigl\|\widetilde{T}_{n,r}(f)-f\bigr\|_\infty,\] where $\widetilde{T}_{n,r}$ denotes the proposed Taylor-accelerated
operator. In addition, we compute the corresponding error
\[E^{\mathrm{Lag}}_{n,r}:= \bigl\|T^{\mathrm{Lag}}_{n,r}(f)-f\bigr\|_\infty,\]
where $T^{\mathrm{Lag}}_{n,r}$ is the Lagrange-based neural network interpolation operator on
irregular grids introduced by Costarelli \emph{et al.}\cite{costarelli2025higher}.

Figures~\ref{fig:error_decay} and~\ref{fig:error_surface} shows the decay  of $E^{\mathrm{Tay}}_{n,r}$ for $r=1,2,3$. As expected, increasing the Taylor order produces a substantial improvement in accuracy, and the error decays faster for larger values of $r$, in agreement with the theoretical rate \[\bigl\|\widetilde{T}_{n,r}(f)-f\bigr\|_\infty \le C\,n^{-(r+1)}.\]

\begin{figure}[H]
\centering
\includegraphics[width=1\textwidth]{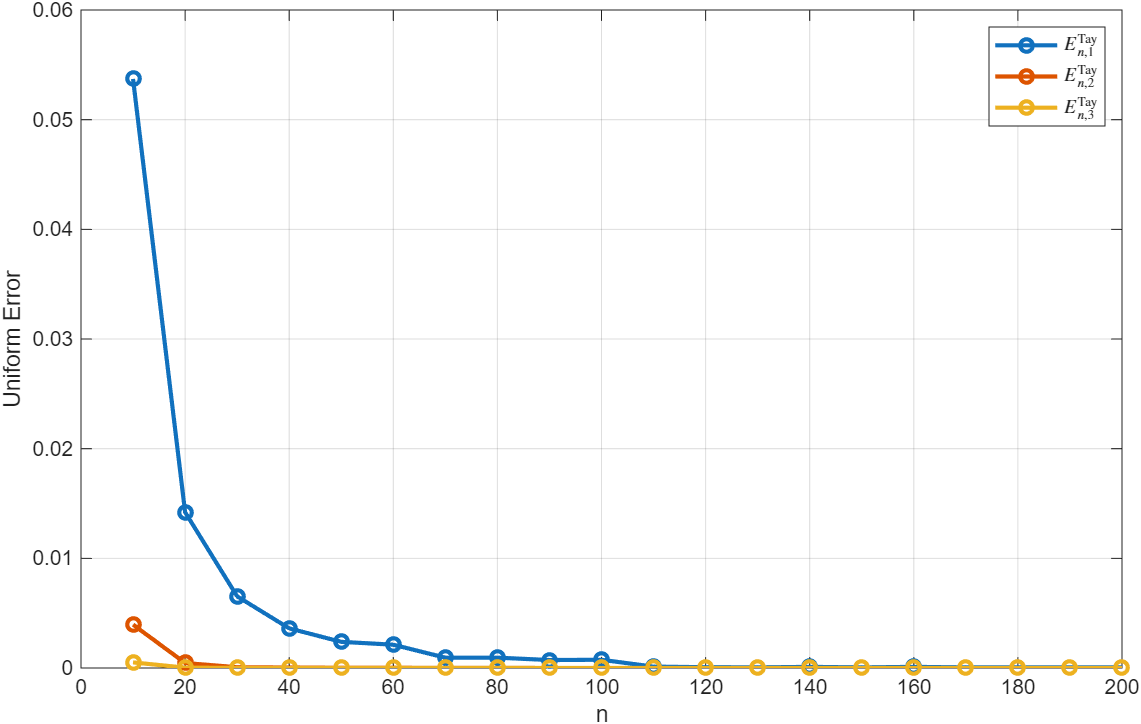}
\caption{Uniform approximation error
$E^{\mathrm{Tay}}_{n,r}=\|\widetilde{T}_{n,r}(f)-f\|_\infty$
as a function of $n$ for Taylor degrees $r=1,2,3$ on nested quasi-uniform irregular grids. Higher values of $r$ yield significantly improved
convergence rates.}
\label{fig:error_decay}
\end{figure}

\begin{figure}[H]
\centering
\includegraphics[width=0.7\textwidth]{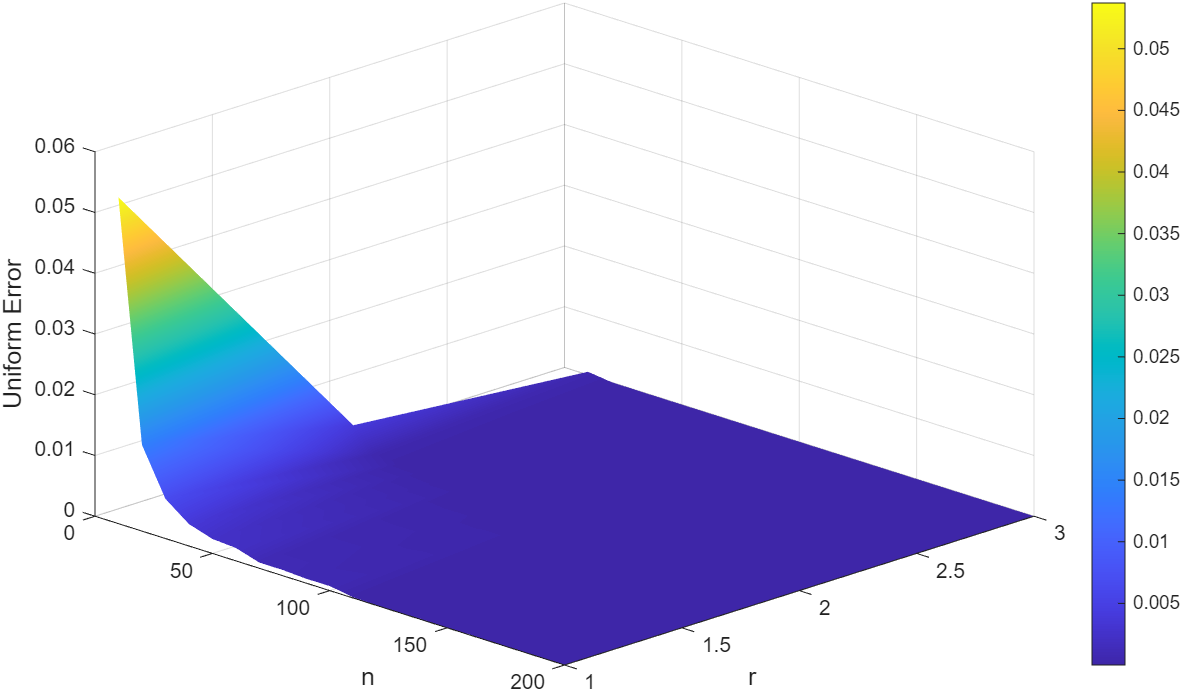}
\caption{Uniform approximation error 
$E^{\mathrm{Tay}}_{n,r}=\|\widetilde{T}_{n,r}(f)-f\|_\infty$ 
displayed as a surface with respect to the grid $n$ and the Taylor degree $r$. 
The error decreases as $n$ increases, while larger Taylor orders yield faster convergence.}
\label{fig:error_surface}
\end{figure}

To highlight the improvement with respect to the Lagrange-based construction, Figures~\ref{fig:comp_r1}-\ref{fig:comp_r3} compare the Taylor and Lagrange interpolation operators for each fixed order $r=1,2,3$. In all cases, the Taylor-accelerated interpolation operator provides smaller errors, demonstrating the effectiveness of the local Taylor enrichment.

\begin{figure}[H]
\centering
\subfloat[$r=1$]{%
\includegraphics[width=0.45\textwidth]{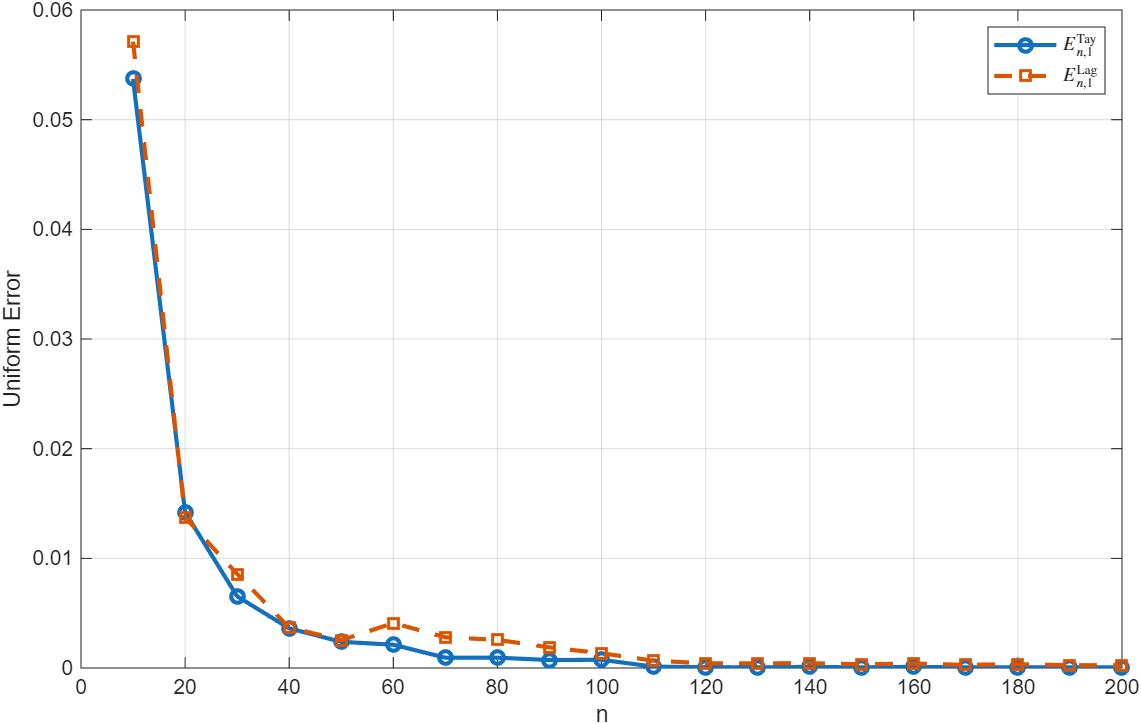}
\label{fig:comp_r1}}
\hfill
\subfloat[$r=2$]{%
\includegraphics[width=0.45\textwidth]{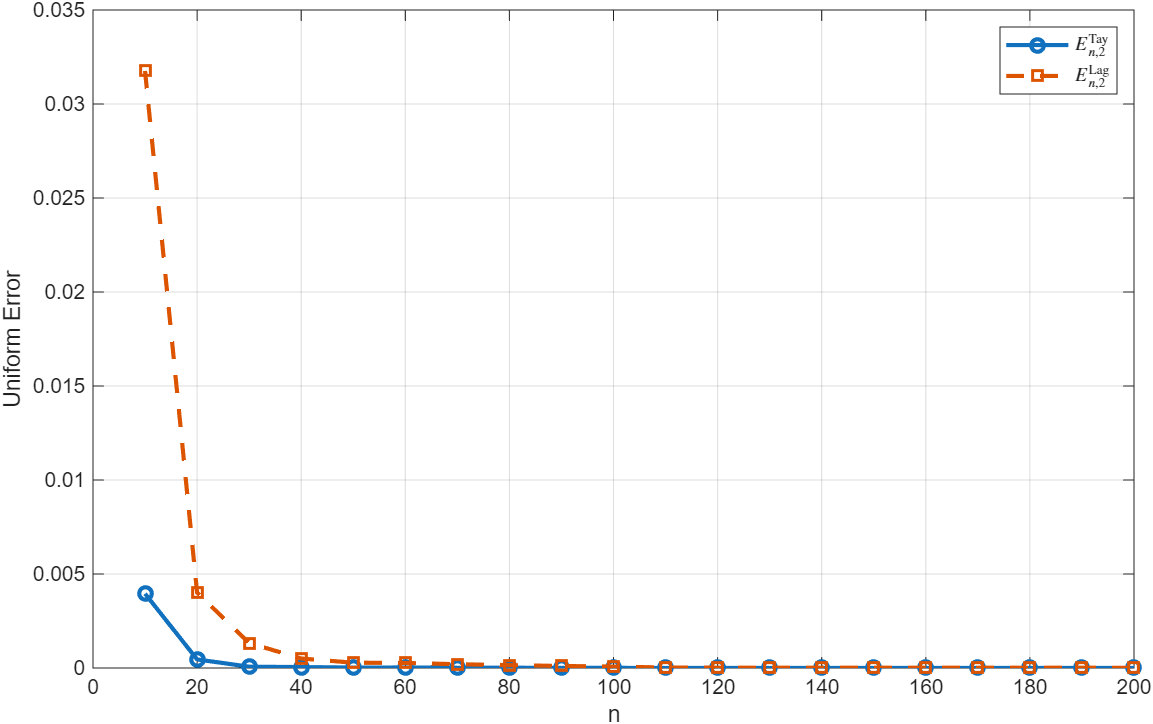}
\label{fig:comp_r2}}

\vspace{0.3cm}

\subfloat[$r=3$]{%
\includegraphics[width=0.55\textwidth]{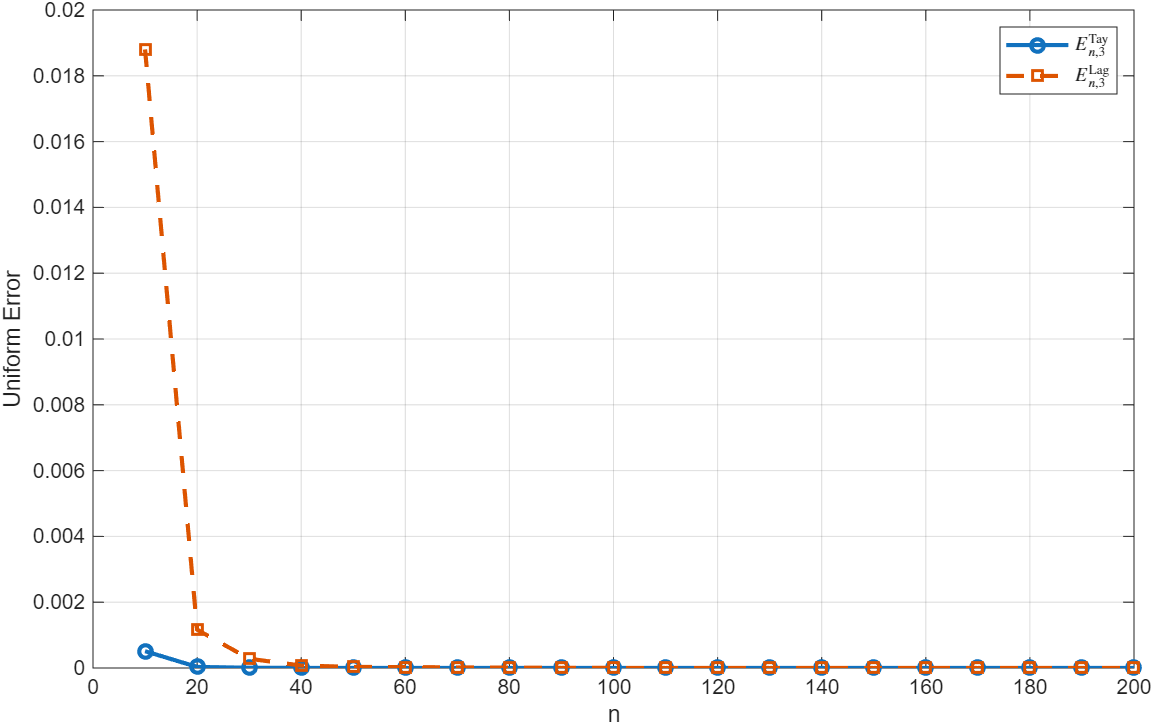}
\label{fig:comp_r3}}

\caption{Error comparison between Taylor and Lagrange neural network interpolation operators for $r=1,2,3$. The Taylor-accelerated operator consistently achieves smaller errors.}
\label{fig:comp_all}

\end{figure}

\subsection{Discussion}
\begin{itemize}
\item The numerical results confirm the theoretical findings of the previous sections. The proposed TNNIOs achieve higher-order approximation rates on quasi-uniform irregular grids, and the convergence order increases with the Taylor degree $r$.

\item  Moreover, the comparative experiments clearly indicate that the Taylor-accelerated construction provides improved convergence and accuracy with respect to the Lagrange-based irregular grid operator of Costarelli \emph{et al.}\cite{costarelli2025higher}, demonstrating the advantages of incorporating local Taylor information into neural network interpolation operators.
\end{itemize}

\subsection{Bivariate Case} \label{sec:numerical2d}

In this subsection, we illustrate the performance of the proposed bivariate TNNIOs $\widetilde{T}_{n,1}^{(2)}$ on an irregular grid, with the aim of validating the theoretical results derived in the previous sections.

\subsection{Test problem}
We consider the rectangular domain \[\Omega = [a_1,b_1]\times[a_2,b_2] = [-2,2]\times[-2,2],\] and the smooth test function \[g(t_1,t_2) = 4 - t_1^2 - t_2^2;, \quad (t_1,t_2)\in\Omega,\] which belongs to $C^\infty(\Omega)$ and exhibits anisotropic behaviour, with distinct variation scales in the two coordinate directions. The Taylor degree is fixed to $r=1$, and irregular quasi-uniform grids with $n=n_1=n_2=50$ sub-intervals are generated in each coordinate direction. 

\subsection{Comparison of exact and reconstructed surfaces}

The precise surface $g(t_1,t_2)$ and its TNNIOs approximation $\widetilde{T}_{n,1}^{(2)} g(t_1,t_2)$ are compared in Figure~\ref{fig:comparison2d}. The smooth decay in the $t_1$-direction and the oscillating behaviour in the $t_2$-direction are both precisely captured by the reconstructed surface, which closely resembles the exact function throughout the domain.

\begin{figure}[H]
\centering
\includegraphics[width=1\textwidth]{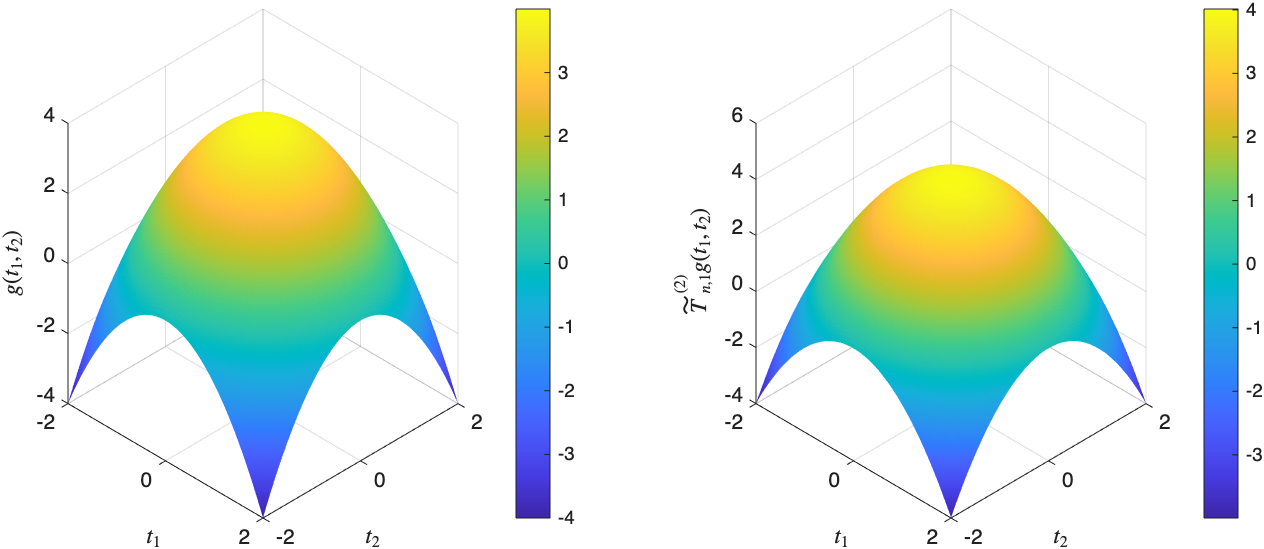}
\caption{The exact surface $g(t_1,t_2)$ (left) and the TNNIOs approximation $\widetilde{T}_{n,1}^{(2)} g(t_1,t_2)$ (right) are shown.}
\label{fig:comparison2d}
\end{figure}

\subsection{Pointwise error analysis}
Figure~\ref{fig:error2d} shows the absolute pointwise error $\lvert \widetilde{T}_{n,1}^{(2)} g(t_1,t_2) - g(t_1,t_2) \rvert$. Both the accompanying contour plot (right) and the surface plot (left) show a uniformly small error distribution, smooth and free of border artefacts and false oscillations.
The operator's localization and the stabilizing influence of the compactly supported activation function are reflected in this behaviour.

\begin{figure}[H]
\centering
\includegraphics[width=1\textwidth]{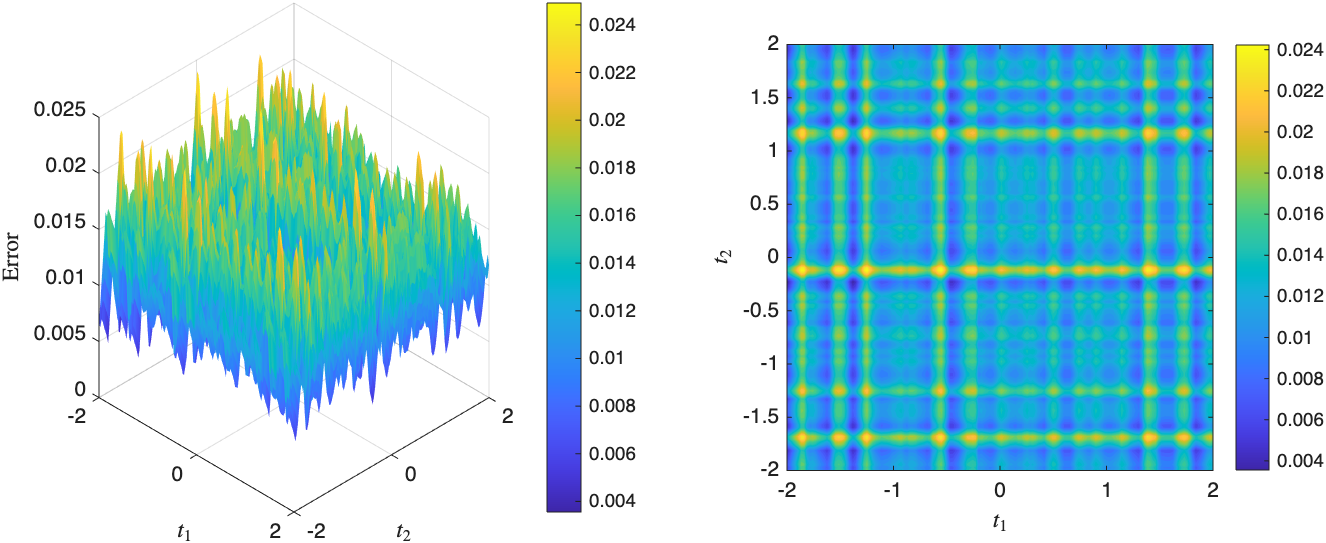}
\caption{The absolute pointwise error $\lvert \widetilde{T}_{n,1}^{(2)} g(t_1,t_2) - g(t_1,t_2) \rvert$ is displayed as a contour plot (right) and a surface plot (left).}
\label{fig:error2d}
\end{figure}

\begin{table}[htbp]
\centering
\begin{tabular}{|c|c|}
\hline
\textbf{Error metric} & \textbf{Value} \\
\hline
Mean Absolute Error (MAE) & $1.260820\times 10^{-2}$ \\
Root Mean Square Error (RMSE) & $1.295315\times 10^{-2}$ \\
Maximum Absolute Error & $2.491778\times 10^{-2}$ \\
Mean Signed Error & $1.260820\times 10^{-2}$ \\
Relative $L^2$ Error & $6.036460\times 10^{-3}$ \\
Relative $L^\infty$ Error & $6.229446\times 10^{-3}$ \\
Normalized RMSE (NRMSE) & $1.619216\times 10^{-3}$ \\
\hline
\end{tabular}
\caption{Numerical error analysis of the proposed TNNIOs for $n=50$.}
\label{tab:numerical_error}
\end{table}

Overall, the theoretical analysis and the numerical results accord completely. Given that $g\in C^{2}(\Omega)$ and $r=1$, the observed accuracy agrees with the expected convergence order \[\|\widetilde{T}_{n,1}^{(2)} g - g\|_\infty = \mathcal{O}\!\left(\max\{h_1,h_2\}^{\,2}\right),\] as predicted by the modulus-of-smoothness framework.

\begin{remark}
It is well established that the ramp-type sigmoidal function constitutes one of the most widely used activation functions in the study of interpolation NNOs; see, for instance, \cite{qian2022rates,lin2026construction,saini2026approximation}. Besides the logistic function, several other activation functions have also been investigated, including the hyperbolic tangent sigmoid and ramp functions, as discussed in \cite{karateke2025some}. Notably, the latter class of functions is closely connected to the celebrated ReLU (Rectified Linear Unit) activation function; see, e.g., \cite{devore2021neural}.
\end{remark}

\section{Application to Derivative-Informed Physical Field Reconstruction}
In this section, we demonstrate how well the proposed TNNIOs reconstructs physical fields from irregular observations. In engineering applications such as fluid dynamics, thermal systems, and electromagnetic signal processing, these issues commonly arise.

\subsection{Physical Field Reconstruction from Irregular Measurements}
We start by considering the reconstruction of spatially distributed physical fields from dispersed data. The suggested operator $\widetilde{T}_{n,r}$ allows for the precise and reliable recovery of continuous fields from discrete observations.

\subsubsection{Reconstruction of the Fluid Velocity Field}
Let $\Omega := [-1,1]^2 \subset \mathbb{R}^2$ and consider a smooth velocity field 
\[\mathbf{u}(t,s) = (u(t,s), v(t,s)),\] representing a steady incompressible potential flow around a circular obstacle of radius $a>0$ centred at the origin.

We define a set of quasi-irregular sampling nodes
\[\Xi_n := \{\xi_k = (t_k,s_k)\}_{k=1}^{n} \subset \Omega \setminus B_a(0),\]
where
\[B_a(0) := \{(t,s) \in \mathbb{R}^2 : t^2 + s^2 < a^2\}\]
denotes the open disk of radius $a>0$ centred at the origin, representing the interior of the cylindrical obstacle.

At each node $\xi_k$, we assume access to function values and first-order derivatives:
\[u_k = u(\xi_k), \quad v_k = v(\xi_k), \quad 
\nabla u(\xi_k), \quad \nabla v(\xi_k).\]

\paragraph{TNNIOs}
We define the first-order Taylor neural network operator acting on $f \in C^1(\Omega)$ as
\[\widetilde{T}_{n,1}^{(2)}(f;(t,s)) := \frac{\displaystyle \sum_{k=1}^{n} \Psi_{\phi_\eta}\!\left(\frac{2\ell\|(t,s)-\xi_k\|}{h_*}\right) \left[f(\xi_k) + \nabla f(\xi_k)\cdot((t,s)-\xi_k)\right]} {\displaystyle \sum_{k=1}^{n} \Psi_{\phi_\eta}\!\left(\frac{2\ell\|(t,s)-\xi_k\|}{h_*}\right)}.\]
The reconstructed velocity field is defined componentwise as
\[\widetilde{\mathbf{u}}_{n,1}(t,s) = \big(\widetilde{T}_{n,1}(u;(t,s)),\,
\widetilde{T}_{n,1}(v;(t,s)) \big).\]

\paragraph{Velocity Magnitude and Visualization}
The reconstructed speed is given by
\[|\widetilde{\mathbf{u}}_{n,1}(t,s)| = \sqrt{\widetilde{u}_{n,1}^2(t,s) + \widetilde{v}_{n,1}^2(t,s)}.\]

\begin{figure}[H]
\centering
\includegraphics[width=0.7\textwidth]{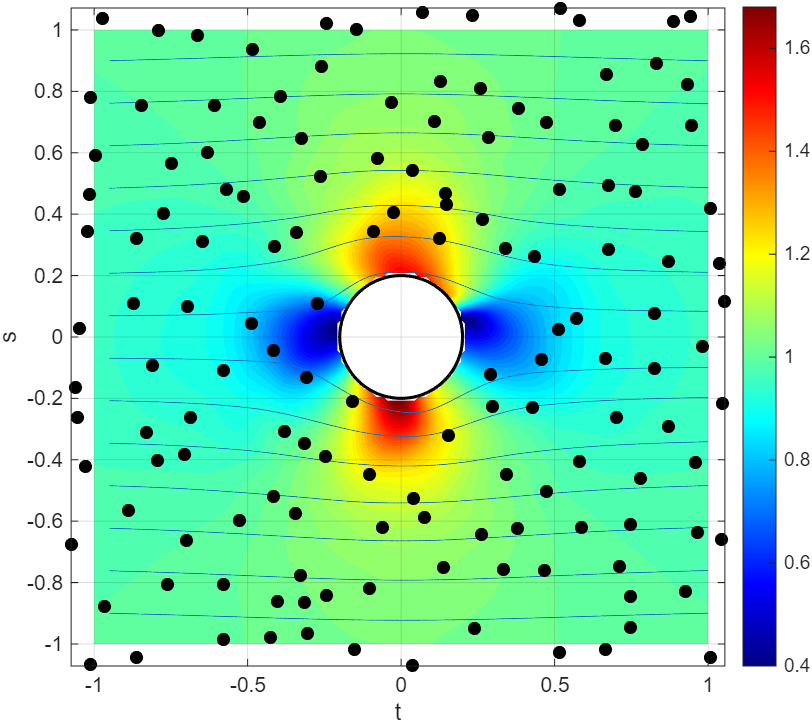}
\caption{Reconstruction of the velocity field using the TNNIOs. Contours represent the velocity magnitude, while streamlines illustrate the flow structure around the cylindrical obstacle for $n=145$.}
\label{fig:velocity_reconstruction}
\end{figure}

\paragraph{Error Metrics}
Let $\Omega_h$ denote the reconstruction grid. We define
\[\|f - \widetilde{T}_{n,1}(f)\|_{\infty} =\max_{(t,s)\in \Omega_h} |f(t,s) - \widetilde{T}_{n,1}(f;(t,s))|,\]
and
\[\mathrm{RMSE}(f) = \left( \frac{1}{|\Omega_h|} \sum_{(t,s)\in\Omega_h}
|f - \widetilde{T}_{n,1}(f)|^2 \right)^{1/2}.\]

\begin{table}[h]
\centering
\begin{tabular}{|c|c|}
\hline
Quantity & Value \\
\hline
$\|u-\widetilde{u}\|_{\infty}$ & $2.612 \times 10^{-1}$ \\
$\|v-\widetilde{v}\|_{\infty}$ & $3.504 \times 10^{-1}$ \\
RMSE$(u)$ & $1.710 \times 10^{-2}$ \\
RMSE$(v)$ & $1.852 \times 10^{-2}$ \\
\hline
\end{tabular}
\caption{Approximation errors of the reconstructed velocity field.}
\label{tab:velocity_errors}
\end{table}

Figure~\ref{fig:velocity_reconstruction} shows that the reconstructed velocity field is smooth and accurately represents the distinct flow pattern despite the irregular distribution of data points. The quantitative error metrics shown in Table~\ref{tab:velocity_errors} provide additional evidence for this visual accuracy.

\subsubsection{Consistency with Governing Equations}
In this subsection, we examine whether the reconstructed field satisfies the incompressibility condition \[\nabla \cdot \mathbf{u} = 0.\]

\paragraph{Discrete Divergence}
Using the reconstructed field $\widetilde{\mathbf{u}}_{n,1} = (\widetilde{u},\widetilde{v})$, we compute
\[\nabla \cdot \widetilde{\mathbf{u}}_{n,1} = \frac{\partial \widetilde{u}}{\partial t} +
\frac{\partial \widetilde{v}}{\partial s}, \]
via numerical differentiation on the reconstruction grid.

\begin{figure}[H]
\centering
\includegraphics[width=0.75\textwidth]{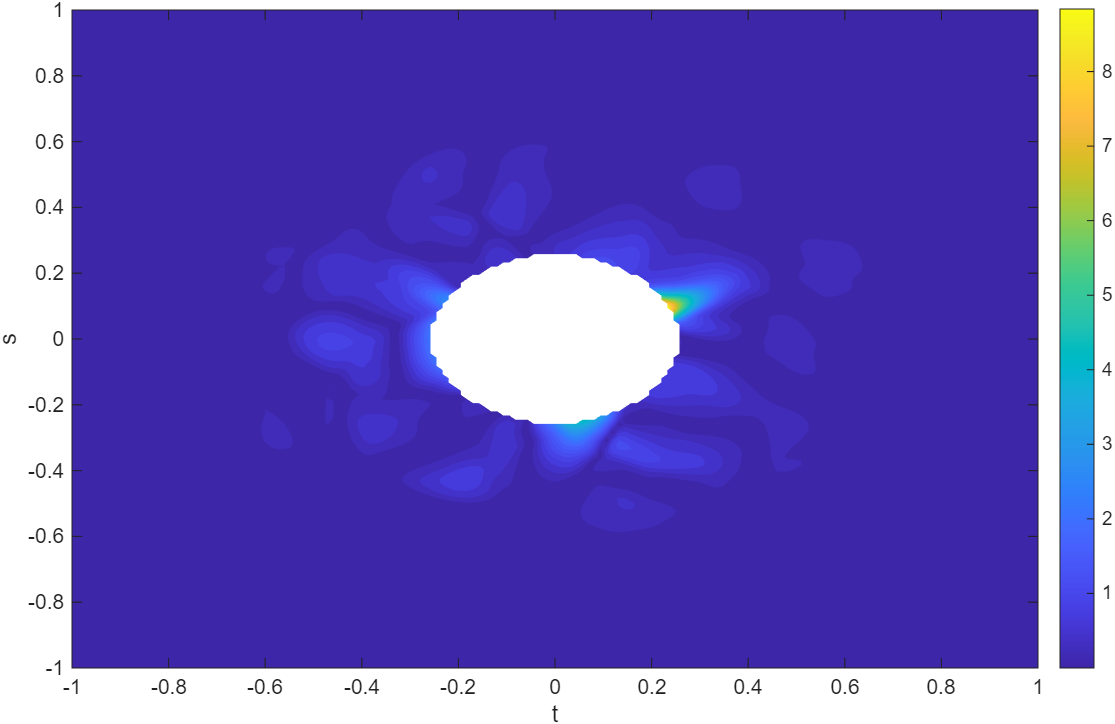}
\caption{Magnitude of the divergence of the reconstructed velocity field. The divergence remains small throughout the domain except near the obstacle boundary.}
\label{fig:divergence}
\end{figure}

\paragraph{Quantitative Measures}
Let $\Omega_h = \{(t_i,s_j)\}_{i,j=1}^{n}$ denote the discrete reconstruction grid. We define
\[ E_{\mathrm{mass}} = \frac{1}{|\Omega_h|} \sum_{(t_i,s_j)\in \Omega_h}
\left| \nabla \cdot \widetilde{\mathbf{u}}(t_i,s_j)
\right|,\]
\[ \Gamma = \oint_{\partial B_{1.1a}} \widetilde{\mathbf{u}}\cdot d\ell,\]
\[ E_{\mathrm{energy}} = \left| \frac{1}{|\Omega_h|} \sum_{(t_i,s_j)\in \Omega_h} |\mathbf{u}(t_i,s_j)|^2 - \frac{1}{|\Omega_h|} \sum_{(t_i,s_j)\in \Omega_h}
|\widetilde{\mathbf{u}}(t_i,s_j)|^2 \right|.\]

\begin{table}[ht]
\centering
\begin{tabular}{|c|c|}
\hline
Quantity & Value \\
\hline
Mass error & $1.148 \times 10^{-1}$ \\
Circulation & $8.541 \times 10^{-2}$ \\
Energy error & $2.294 \times 10^{-4}$ \\
\hline
\end{tabular}
\caption{Consistency of the reconstructed field with physical constraints.}
\label{tab:governing_consistency}
\end{table}

The reconstructed velocity field roughly satisfies the incompressibility criterion, as seen in Figure~\ref{fig:divergence}, where the divergence remains small across the interior of the domain. This observation is further supported by the quantitative measures reported in Table~\ref{tab:governing_consistency}. Due to the lack of sampling points inside the cylinder, the compact support of the activation function, and numerical differentiation effects, localized deviations are observed near the obstruction boundary.

Despite not explicitly enforcing the divergence-free condition, the operator produces a nearly incompressible field. The deviations near the boundary arise from activation truncation and the lack of sampling inside the obstacle.

\subsubsection{Pressure Field Recovery}
We now reconstruct the pressure field using the Bernoulli principle.

\paragraph{Bernoulli Relation}
For steady incompressible potential flow, \[p(t,s) + \frac{1}{2} |\mathbf{u}(t,s)|^2 = C.\]
Thus,
\[p(t,s) = C - \frac{1}{2} |\mathbf{u}(t,s)|^2.\]

\paragraph{TNNIOs-based Pressure Reconstruction}
Using $\widetilde{\mathbf{u}}_{n,1}$, we define
\[\widetilde{p}_{n}(t,s) = - \frac{1}{2} \left( \widetilde{u}_{n,1}^2(t,s) +
\widetilde{v}_{n,1}^2(t,s) \right),\]
followed by normalization to remove the constant offset.

\begin{figure}[H]
\centering
\includegraphics[width=0.8\textwidth]{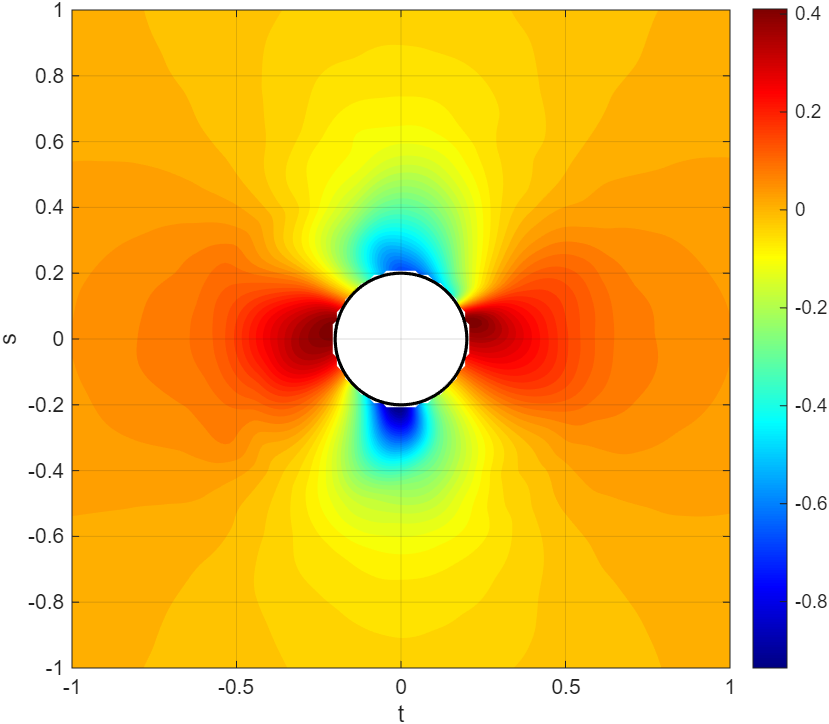}
\caption{Recovered pressure field. Low-pressure regions correspond to high-velocity zones, while high-pressure regions occur near stagnation points.}
\label{fig:pressure}
\end{figure}

\paragraph{Pressure Coefficient}
We define
\[C_p(t,s) = 1 - |\widetilde{\mathbf{u}}_{n,1}(t,s)|^2.\]

\begin{table}[ht]
\centering
\begin{tabular}{|c|c|}
\hline
Quantity & Value \\
\hline
Pressure coefficient error & $5.852 \times 10^{-1}$ \\
\hline
\end{tabular}
\caption{Error in pressure reconstruction.}
\label{tab:pressure_error}
\end{table}

Figure~\ref{fig:pressure} displays the reconstructed pressure field. Low-pressure regions correspond to high-velocity zones, whereas high-pressure regions are near stagnation points. For potential flow around a cylinder, this qualitative behaviour is consistent with classical fluid dynamics principles. Table~\ref{tab:pressure_error} provides additional information to corroborate this observation.

The recovered pressure field is consistent with classical fluid dynamics, capturing high-pressure stagnation regions and low-pressure acceleration zones. Notably, this reconstruction is achieved directly from the operator without solving additional governing equations.

\subsubsection{Comparison with Physics-Informed Neural Networks}

We briefly compare the proposed TNNIOs with the class of Physics-Informed Neural Networks (PINNs), which are widely used for data-driven reconstruction of fluid fields.

\begin{itemize}

\item Physics-Informed Neural Networks (PINNs) aim to approximate the solution $(\mathbf{u},p)$ of governing equations by minimizing a composite loss functional of the form
\[ \mathcal{L}(\theta) = \mathcal{L}_{\text{data}}(\theta) + \lambda\, \mathcal{L}_{\text{physics}}(\theta), \]
where $\theta$ denotes network parameters, $\mathcal{L}_{\text{data}}$ enforces agreement with measurements, and $\mathcal{L}_{\text{physics}}$ penalizes violations of the governing PDEs. The reconstruction is obtained via iterative optimization.

\item In contrast, the proposed method is based on an explicit operator of the form
\[
\widetilde{T}_{n,1}(f;t) = \frac{ \sum\limits_{k=1}^{n} \phi_\eta\!\left(\frac{2\ell}{h_*}(t-\xi_k)\right) \Big[ f(\xi_k) + \nabla f(\xi_k)\cdot(x-\xi_k)
\Big] }{ \sum\limits_{k=1}^{n} \phi_\eta\!\left(\frac{2\ell}{h_*}(t-\xi_k)\right)},\]
which provides a direct,  reconstruction from irregular data. This eliminates the need for parameter optimization and guarantees deterministic outputs.

\item A key distinction lies in the treatment of derivatives. In PINNs, derivatives such as $\nabla \cdot \mathbf{u}$ are obtained via automatic differentiation of the trained neural network, which may be sensitive to noise and optimization artifacts. In contrast, the proposed operator incorporates local Taylor expansions, yielding analytically structured approximations that enhance the stability and accuracy of derivative computations.

\item The proposed operator exhibits exact or near-exact interpolation at measurement sites due to its localized construction, whereas PINNs typically provide approximate data fitting depending on the balance between loss terms. This property is particularly advantageous in applications requiring strict adherence to observed data.

\item It is important to note that PINNs explicitly enforce physical constraints through the loss functional, which can be advantageous in extremely irregular or noisy data regimes. In contrast, the proposed operator achieves approximate physical consistency (e.g., near divergence-free behavior) through its approximation-theoretic structure, without explicitly enforcing the governing equations.


\end{itemize}

\subsection{Temperature Field Reconstruction from Irregular Sensors}
In this experiment, we demonstrate the effectiveness of the proposed TNNIOs for reconstructing a two-dimensional temperature field from irregular sensor measurements.

\paragraph{Problem Setup}
Let $\Omega := [0,1]^2 \subset \mathbb{R}^2$, and consider the temperature field
\[T(t,s) = \sin(\pi t)\sin(\pi s), \quad (t,s) \in \Omega,\] which is smooth, vanishes on the boundary $\partial \Omega$, and attains its maximum at the centre of the domain.

Let  \[\Xi_n := \{\xi_k = (t_k,s_k)\}_{k=1}^n \subset \Omega\] denote a set of quasi-irregular sensor locations. The measurements are assumed to be contaminated with additive Gaussian noise:
\[T_k = T(\xi_k) + \varepsilon_k, \quad \varepsilon_k \sim \mathcal{N}(0,\sigma^2) \ \text{ with } \sigma^2= 0.01.\]
In addition, first-order derivatives are available:
\[T_{t,k} = \partial_t T(\xi_k), \quad T_{s,k} = \partial_s T(\xi_k).\]
Now, we define the reconstruction operator TNNIOs as
\[\widetilde{T}_{n,1}^{(2)}(t,s) = \frac{ \displaystyle \sum_{k=1}^n 
\Psi_{\phi_\eta}\!\left(\frac{2\ell\|(t,s)-\xi_k\|}{h_*}\right)
\left[ T_k + T_{t,k}(t-t_k) + T_{s,k}(s-s_k)
\right] }{ \displaystyle \sum_{k=1}^n 
\Psi_{\phi_\eta}\!\left(\frac{2\ell\|(t,s)-\xi_k\|}{h_*}\right)
}.\]

\paragraph{Reconstruction Results}
The reconstructed temperature field $\widetilde{T}_{n,1}$ is shown in Figure~\ref{fig:temp_reconstruction}. The operator accurately captures the global structure of the field, including the central peak and smooth decay toward the boundary, despite the irregular sampling and measurement noise for $n=200$.

\begin{figure}[H]
\centering
\includegraphics[width=0.6\textwidth]{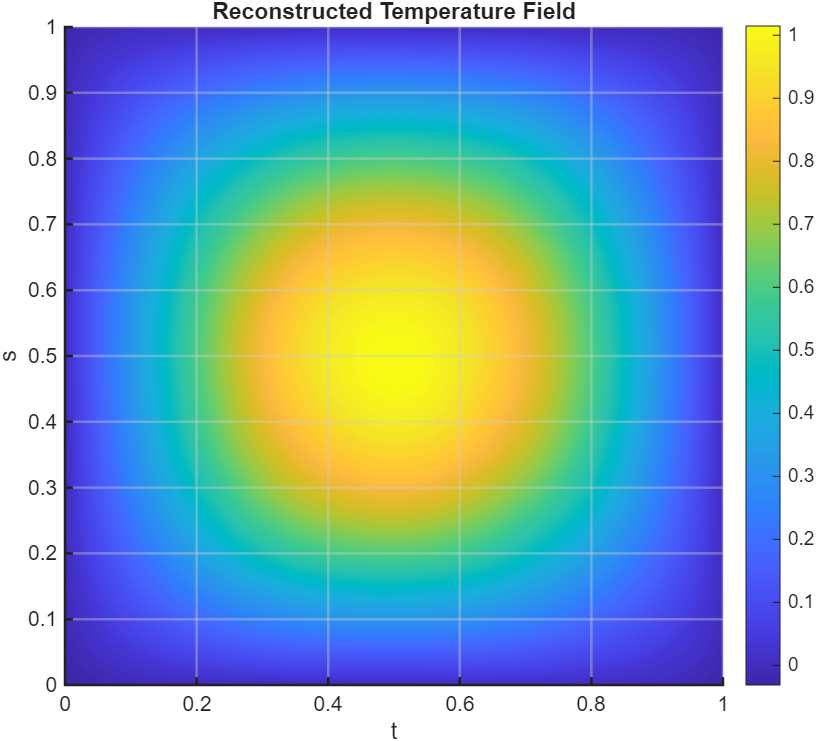}
\caption{Reconstructed temperature field using the adaptive Taylor neural operator.}
\label{fig:temp_reconstruction}
\end{figure}
\paragraph{Error Analysis}
We define the pointwise error \[e(t,s) := |\widetilde{T}_{n,1}(t,s) - T(t,s)|.\]
The global error measures are given by
\[\|e\|_{\infty} = \max_{(t,s)\in\Omega_h} e(t,s), \quad
\mathrm{RMSE} = \left( \frac{1}{|\Omega_h|} \sum_{(t,s)\in\Omega_h} e(t,s)^2
\right)^{1/2}.\]

\begin{figure}[H]
\centering
\includegraphics[width=0.7\textwidth]{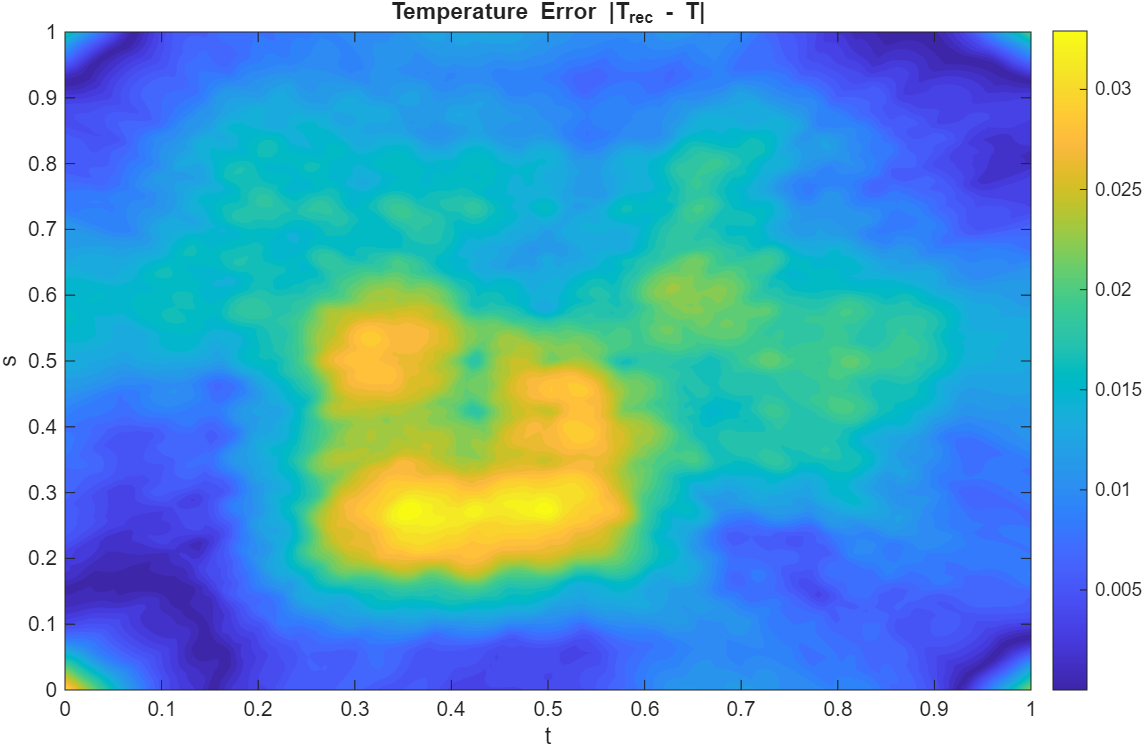}
\caption{Pointwise reconstruction error over the domain.}
\label{fig:temp_error}
\end{figure}

Figure~\ref{fig:temp_error} illustrates that the approximation error stays uniformly low throughout the domain, indicating strong robustness against noise and irregular sampling.

\paragraph{Heat Flux Recovery}
To assess physical consistency, we compute the heat flux using Fourier's law:
\[\mathbf{q}(t,s) = -k \nabla T(t,s).\]
Using the reconstructed field, we define
\[\widetilde{\mathbf{q}}(t,s) = - k \nabla \widetilde{T}_{n,1}(t,s).\]
The recovered heat flux field is displayed in Figure~\ref{fig:heat_flux}. The vector field is smooth and exhibits the expected symmetric structure induced by the underlying temperature distribution.

\begin{figure}[H]
\centering
\includegraphics[width=0.75\textwidth]{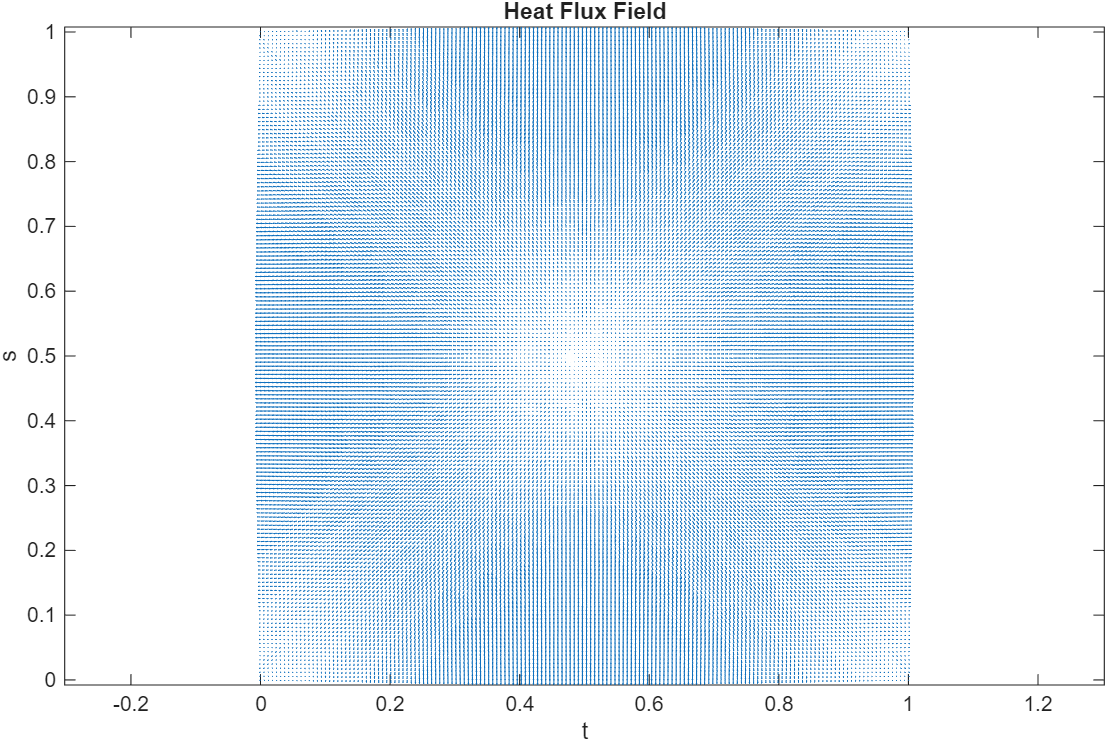}
\caption{Recovered heat flux field from the reconstructed temperature.}
\label{fig:heat_flux}
\end{figure}

\paragraph{Physical Metrics}
Let $\Omega_h = \{(t_i,s_j)\}_{i,j=1}^{n}$ denote the reconstruction grid. We evaluate the following quantities:
\[
\text{Flux error}(t_i,s_j) = \left| |\widetilde{\mathbf{q}}(t_i,s_j)| -
|\mathbf{q}(t_i,s_j)| \right|, \]
\[E_{\mathrm{energy}} = \left| \frac{1}{|\Omega_h|} \sum_{(t_i,s_j)\in \Omega_h} |\widetilde{\mathbf{q}}(t_i,s_j)|^2 - \frac{1}{|\Omega_h|}
\sum_{(t_i,s_j)\in \Omega_h} |\mathbf{q}(t_i,s_j)|^2
\right|,\]
\[\langle |\mathbf{q}| \rangle = \frac{1}{|\Omega_h|}
\sum_{(t_i,s_j)\in \Omega_h} |\mathbf{q}(t_i,s_j)|.\]

\begin{table}[H]
\centering
\begin{tabular}{|c|c|}
\hline
Quantity & Value \\
\hline
Supremum error & $3.374 \times 10^{-2}$ \\
RMSE & $1.497 \times 10^{-2}$ \\
Mean flux error & $6.044 \times 10^{-2}$ \\
Max flux error & $4.939 \times 10^{-1}$ \\
Energy (reconstructed) & $5.093$ \\
Energy (exact) & $4.935$ \\
Energy error & $1.580 \times 10^{-1}$ \\
Mean $|\mathbf{q}|$ (reconstructed) & $2.165$ \\
Mean $|\mathbf{q}|$ (exact) & $2.127$ \\
\hline
\end{tabular}
\caption{Reconstruction accuracy and physical consistency metrics.}
\label{tab:temp_metrics}
\end{table}

\paragraph{Discussion}
The results demonstrate that the proposed TNNIOs provides a robust framework for reconstructing temperature fields from noisy and irregular data. The method achieves low approximation error while preserving essential physical properties such as smoothness and gradient structure.

Moreover, the accurate recovery of derived quantities, including heat flux and energy, indicates that the operator captures not only pointwise behaviour but also global physical characteristics as shown in Table~\ref{tab:temp_metrics}. This highlights its potential for applications in data-driven thermal analysis, sensor-based monitoring, and scientific computing.

\subsection{Electromagnetic Radiation Pattern Reconstruction}
In this subsection, we investigate the reconstruction of antenna radiation patterns on the unit sphere using the proposed operator. This experiment demonstrates the applicability of the operator to practical electromagnetic problems involving irregular and noisy angular measurements.

\paragraph{Angular Domain and Radiation Models}
Let $(\theta,\phi) \in [0,\pi]\times[0,2\pi]$ denote the spherical angular coordinates. We consider two standard dipole radiation models:
\[R_{\text{ideal}}(\theta,\phi) = |\sin \theta|,\]
and the half-wave dipole model
\[R_{\text{real}}(\theta,\phi) = \left| \frac{\cos\!\left(\frac{\pi}{2}\cos\theta\right)}{\sin\theta} \right|.\]
Both models are independent of $\phi$ but are represented over the full angular domain for visualization purposes.

\paragraph{Irregular Sampling and Noisy Measurements}
Let 
\[\Theta_{n_1} := \{\theta_k\}_{k=1}^{n_1} \subset [0,\pi], 
\quad \Phi_{n_2} := \{\phi_l\}_{l=1}^{n_2} \subset [0,2\pi] \]
be quasi-uniform grids perturbed by small random noise. The sampling set is given by the grid
\[\Xi_{n_1,n_2} := \{(\theta_k,\phi_l)\}_{k=1,l=1}^{n_1,n_2}.\]
The corresponding measurements are
\[R_{k,l} = R(\theta_k,\phi_l) + \varepsilon_{k,l}, \quad \varepsilon_{k,l} \sim \mathcal{N}(0,\sigma^2) \ \text{ with } \sigma^2= 0.05,\] representing noisy observations of the radiation intensity.

Define the mesh sizes
\[h_1 := \max_{k} (\theta_{k+1} - \theta_k), \quad h_2 := \max_{l} (\phi_{l+1} - \phi_l).\]

\paragraph{TNNIOs}
Let $\phi_\eta:\mathbb{R}\to\mathbb{R}$ be a compactly supported activation function generated via a sigmoidal construction. The reconstruction operator is defined as
\[\widetilde{R}_{n_1,n_2}(\theta,\phi) = \frac{ \displaystyle \sum_{k=1}^{n_1} \sum_{l=1}^{n_2}
\phi_\eta\!\left(\frac{2\ell}{h_1}(\theta-\theta_k)\right) \phi_\eta\!\left(\frac{2\ell}{h_2}(\phi-\phi_l)\right) \, R_{k,l} }{\displaystyle \sum_{k=1}^{n_1} \sum_{l=1}^{n_2} \phi_\eta\!\left(\frac{2\ell}{h_1}(\theta-\theta_k)\right)
\phi_\eta\!\left(\frac{2\ell}{h_2}(\phi-\phi_l)\right)}.\]

\paragraph{Reconstruction Results}
Figure~\ref{fig:ideal_dipole} presents the reconstruction of the ideal dipole radiation pattern. The operator accurately recovers the toroidal structure with negligible error, confirming its consistency and high approximation capability in the absence of noise for $n_1=n_2=40$.

\begin{figure}[H]
\centering
\includegraphics[width=1\textwidth]{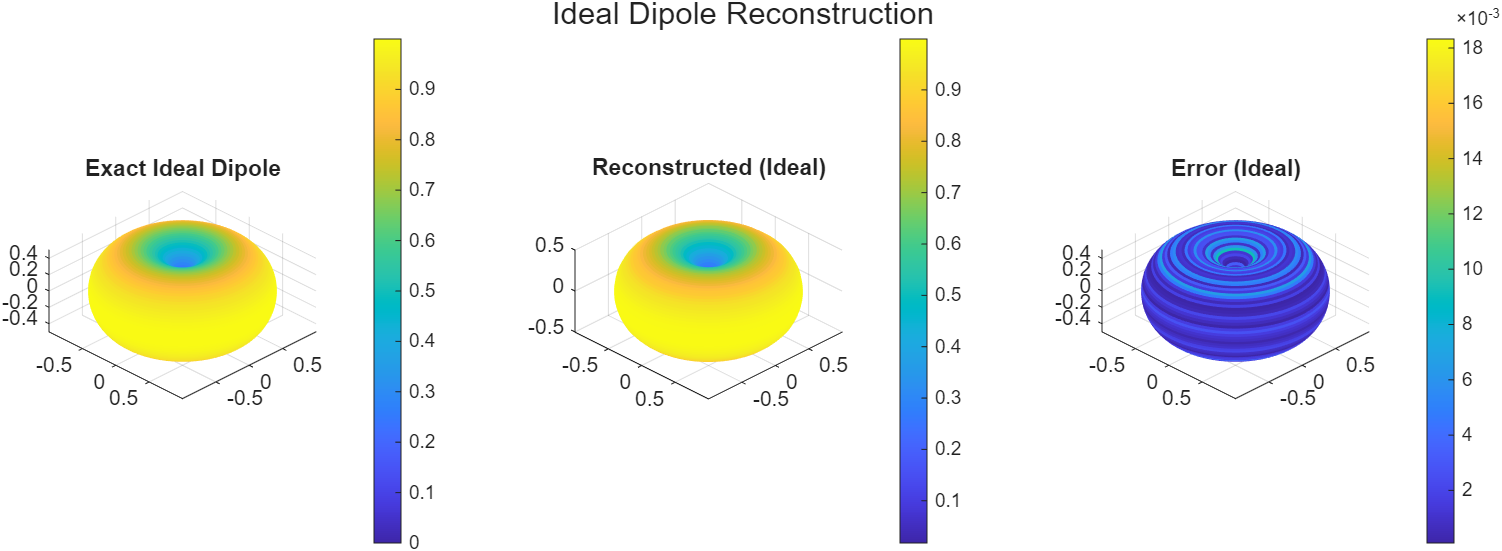}
\caption{Reconstruction of the ideal dipole radiation pattern using the TNNIOs.}
\label{fig:ideal_dipole}
\end{figure}

\paragraph{Reconstruction under Noise}
In the realistic scenario, noisy measurements of the half-wave dipole model are used. Figure~\ref{fig:real_dipole} shows that the operator successfully reconstructs the dominant radiation structure while suppressing noise-induced perturbations $n_1=n_2=40$.

\begin{figure}[H]
\centering
\includegraphics[width=1\textwidth]{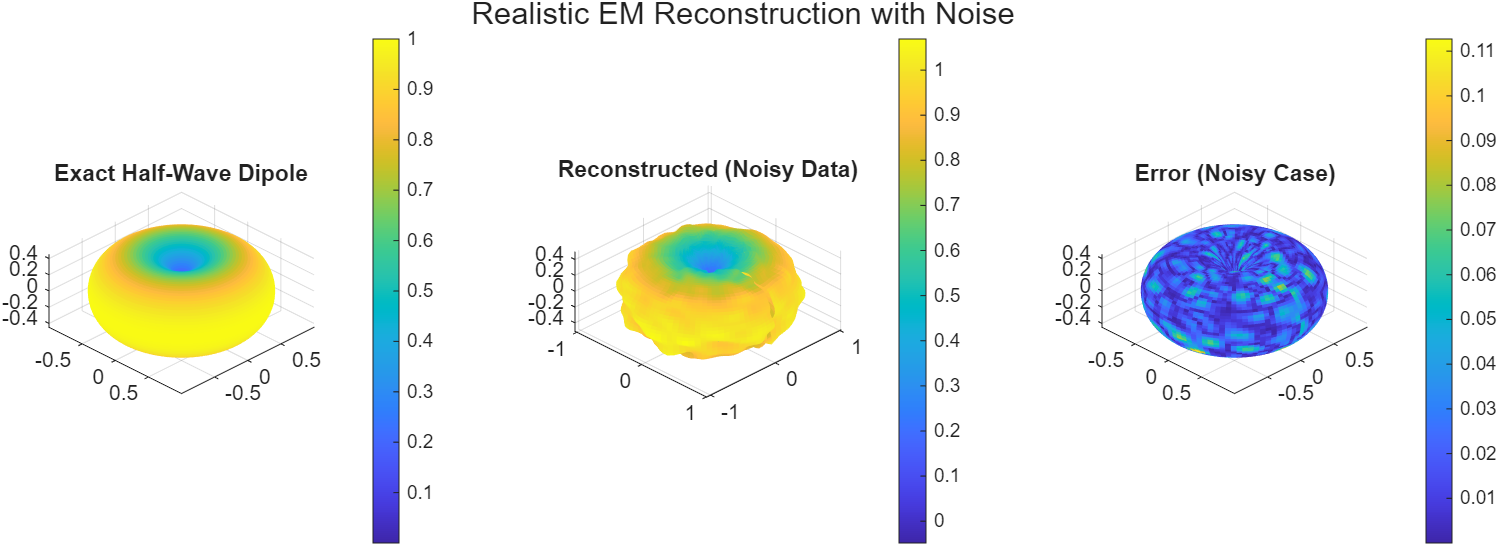}
\caption{Reconstruction of the half-wave dipole radiation pattern from noisy measurements.}
\label{fig:real_dipole}
\end{figure}

\paragraph{Error Measure}
We define the pointwise reconstruction error
\[E(\theta,\phi) = \left| \widetilde{R}_{n_1,n_2}(\theta,\phi) - R(\theta,\phi) \right|.\]
The numerical results indicate that the reconstruction error remains small and spatially localized even in the presence of noise.

\paragraph{Discussion}
The proposed TNNIOs provide a stable and accurate framework for reconstructing electromagnetic radiation patterns from irregular and noisy angular measurements. The method preserves key physical features such as smoothness, symmetry, and directional radiation characteristics.

Moreover, the activation mechanism induces an inherent denoising effect, leading to smooth reconstructions without requiring any optimization procedure. These properties make the approach particularly suitable for applications in antenna characterization, inverse scattering, and data-driven electromagnetic modeling.

\section{Conclusion}

In this research, a new class of interpolation operators for Taylor-accelerated neural networks is defined on quasi-uniform irregular grids. The given framework effectively leverages the smoothness of the target function by incorporating local Taylor expansions into NNOs, thereby improving approximation accuracy and stability.

Theoretically, we demonstrated that the proposed operators are uniformly constrained, well-defined, and maintain accurate interpolation at the sample nodes. They also reproduce algebraic polynomials up to a certain degree and satisfy Jackson-type direct approximation estimates in terms of higher-order moduli of smoothness. Consequently, for sufficiently smooth functions, the operators achieve higher-order convergence rates of order $r+1$ with asymptotic behaviour 
\[ \|\widetilde{T}_{n,r}(\chi)-\chi\|_\infty = \mathcal O\bigl(n^{-(r+1)}\bigr), \quad n\to\infty.\]

The theoretical results are validated by numerical tests, which also show that the proposed approach consistently outperforms current neural network interpolation methods on irregular grids, including Lagrange-type designs. Specifically, the results show how adding local Taylor information to the neural operator framework significantly improves accuracy

The suggested method shows significant practical importance for engineering applications beyond approximation accuracy. The capacity of the method to restore smooth, physically consistent fields while enabling the steady computation of derived values is well demonstrated by the reconstruction of fluid velocity and pressure fields from irregular observations. The efficiency of the suggested operators in managing noisy, and irregular sensor data is further demonstrated by additional tests on temperature and electromagnetic signal reconstruction.


The development of multiscale and adaptive methods on irregular grids, as well as stochastic and fuzzy extensions of the proposed operators, is a future research avenue. Incorporating explicit physical constraints into the operator building to create hybrid physics-informed neural operator models is another potential approach.

\textbf{\large Declaration of competing interest}

The author declares that there are no known competing financial interests or personal relationships that could have influenced the work reported in this paper.

\noindent\textbf{Data availability}

No datasets were generated or analyzed during the current study.

\noindent\textbf{Funding}

The author declares that no funds, grants, or other financial support were received for the research, authorship, or publication of this work.

\textbf{\large Acknowledgement}

The author acknowledges the use of ChatGPT (OpenAI) for assistance and minor language support in preparing the manuscript.

\bibliographystyle{elsarticle-num}
\bibliography{Ref}
\end{document}